\documentclass[11pt]{article}
\usepackage{graphicx}
\usepackage{amsmath,amssymb,amsthm,amsfonts}
\usepackage{color}
\usepackage{cite}
\newtheorem{thm}{Theorem}[section]

\newtheorem{lem}[thm]{Lemma}

\theoremstyle{definition}
\newtheorem{defn}{Definition}[section]
\theoremstyle{remark}
\usepackage{appendix}
\newtheorem{rem}{Remark}[section]

\numberwithin{equation}{section}

\DeclareMathSymbol{\C}{\mathalpha}{AMSb}{"43}

\newcommand{\R}{{\mathbb{R}}}

\newcommand{\inte}{\int_{\mathbb{R}^d}}

\newcommand{\Tr}{\mathrm{Tr}}

\newcommand{\fS}{\mathfrak{S}}

\def\R{{\mathbb R}}
\def\C{{\mathbb C}}

\newcommand{\bsub}{\begin{subequations}}
\newcommand{\esub}{\end{subequations}$\!$}

\allowdisplaybreaks[4]

\begin{document}
\title{Ground States of Nonlinear Fermionic Systems: From Power-Law Interactions to Logarithmic Sobolev Inequality
}

\author{
Bin Chen\thanks{Academy of Mathematics and Systems Science, Chinese Academy of Sciences, Beijing 100190, P. R. China. B. Chen  is partially supported by NSF of China (Grant 12501151).
Email:  binchen@amss.ac.cn.},  \
Yujin Guo\thanks{School of Mathematics and Statistics,  Key Laboratory of Nonlinear Analysis $\&$ Applications (Ministry of Education), Central China Normal University, Wuhan 430079, P. R. China. Y.  Guo is partially supported by  NSF of China (Grants 12225106 and 12371113) and National Key R $\&$ D Program of China (Grant 2023YFA1010001). Email: yguo@ccnu.edu.cn. },\
Yong Luo\thanks{School of Mathematics and Statistics, and Hubei Key Laboratory of Mathematical Sciences,  Central China Normal University, Wuhan 430079, P. R. China. Y. Luo is partially supported by NSF of China (Grant  12571119). Email: yluo@ccnu.edu.cn.},
\, and\,
Zhenya Yan\thanks{School of Mathematics and Information Science, Zhongyuan University of Technology, Zhengzhou 450007, P. R. China; State Key Laboratory of Mathematical Sciences, Academy of Mathematics and Systems Science, Chinese Academy of Sciences, Beijing 100190, P. R. China; School of Mathematical Sciences, University of Chinese Academy of Sciences, Beijing 100049, P. R. China.  Z. Yan is partially supported by NSF of China (Grant 12471242) and National Key R \& D Program of China (Grant 2024YFA1013101).  Email: zyyan@mmrc.iss.ac.cn.}
}
\date{\today}

\smallbreak \maketitle
	
\begin{abstract}
We consider ground states of a two-component  logarithmic  fermionic system in $\R^d$, where $d\ge 1$ is arbitrary. We prove that  up to  translations and scalings, ground states of the logarithmic system are the $L^\infty $-limits of ground states for a two-component   $2p-1$ power-law  fermionic   system  as $p\searrow 1$. As a byproduct, we also establish a sharp logarithmic Sobolev inequality for orthonormal functions in \(\mathbb{R}^d\), whose optimizers are, up to  scalings, the minimizers of a constraint variational problem associated with the logarithmic system.
\end{abstract}
	
\vskip 0.05truein
	

\vskip 0.2truein
	

\section{Introduction}

In quantum mechanics, a system of \(N\)  identical fermions can be described  by an energy functional of normalized antisymmetric \(N\)-particle wave functions.  The infimum of this energy functional is called the ground-state energy of the system, and any minimizer, when it exists, is called a ground state of the system. 
In the effective fermionic model considered in
\cite{i, Lieb1981} and the references therein, the corresponding ground-state problem can be formulated in terms of Slater determinants, which are generated by orthonormal one-particle orbitals. Especially, Gontier, Lewin, and Nazar
\cite{i} studied  the following  constraint  variational problem
\begin{equation}\label{ep}
	\begin{split}
		E_N(p):=&\inf\Big\{\text{Tr}(-\Delta\gamma)-\frac{1}{p}\inte\rho_\gamma^pdx:\  \gamma\in\mathcal{P}_N\Big\},\ \,  N\in\mathbb{N}^+,\,\ p>1,
\end{split}\end{equation}
where
\begin{equation*}\label{1.4a}
	\text{Tr}(-\Delta\gamma):= \sum_{i=1}^N\inte|\nabla u_i|^2dx,\ \ \   \rho_\gamma:=\sum_{i=1}^N|u_i|^2,
\end{equation*}
and
\begin{align}\label{pl}
	\mathcal{P}_N:=&\Big\{ \gamma:\,  \gamma=\sum_{i=1}^N|u_i\rangle\langle u_i|,\ u_i\in H^1(\R^d, \C),\ \langle u_i,u_j\rangle_{L^2}=\delta_{ij}\Big\}
\end{align}
denotes the set of  one-body density matrices associated with the Slater determinant. 

It was proved in \cite{i} that for any $d\geq1$,  there exists a constant $1<p_c(d)<\min\big(2,\,  1+2/d\big)$ such that the problem $E_N(p)$ admits  a minimizer ({\em i.e., ground state}) $\gamma_p=\sum_{i=1}^N|u_{ip}\rangle\langle u_{ip}|$. Here $u_p=(u_{1p}, \cdots, u_{Np})$ 
satisfies the following $2p-1$ power-law fermionic Schr\"odinger system
\begin{align}\label{eqv}
	\arraycolsep=1.5pt
	\left\{\begin{array}{lll}
		-\Delta u_{i}-\big(\sum_{j=1}^N|u_j|^2\big)^{p-1}u_{i}=\mu_{i}u_{i}\ \   \text{in}\ \, \R^d, \\[3mm]
		\langle u_i,u_j\rangle_{L^2}=\delta_{ij},\ \   i, j=1,\cdots, N,
	\end{array}\right.
\end{align}
where  $\mu_1<\mu_2\leq\cdots\leq\mu_N<0$ are the  Lagrange multipliers and  the  $N$ first eigenvalues (counted with multiplicity) of the operator $-\Delta -\big(\sum_{j=1}^N|u_j|^2\big)^{p-1}$ on $L^2(\R^d,\, \C)$.
We also refer the reader to \cite{ge,ii,cg1}  for more analytical results of the problem $E_N(p)$.

Particularly,  the  problem \eqref{ep} with  \(N=1\) reduces to the following classical mass constraint  variational problem
\[
E_1(p)
=\left\{\int_{\mathbb R^d}|\nabla u|^2\,dx
-\frac1p\int_{\mathbb R^d}|u|^{2p}\,dx:\  u\in H^1(\mathbb R^d,\mathbb C),\ \|u\|_2^2=1\right\},
\]
which is often used (cf. \cite{Cazenave, Rabinowitz}) to describe nonlinear optics, attractive Bose-Einstein condensates, and some other physical phenomena.
In the mass subcritical range \(1<p<1+2/d\), the problem $E_1(p)$ admits a unique (up to translations and a constant phase) minimizer $u>0$,
which satisfies the corresponding Euler--Lagrange equation
\begin{align}\label{nls}
	-\Delta u-|u|^{2p-2}u=\mu u\ \ \text{in}\ \, \mathbb R^d
\end{align}
for some Lagrange multiplier $\mu\in\R$. 
One can observe that
\begin{align}\label{M:nls}
	\frac{|u|^{2p-2}-1}{p-1}u\to u\log|u|^2\ \ \ \text{as}\ \ p\searrow1.
\end{align}
This yields formally that after a suitable scaling and a corresponding normalization of $\mu$,  the power-law Schr\"odinger equation \eqref{nls} is related as \(p\searrow1\) to the logarithmic Schr\"odinger equation
\begin{align}\label{lognls1}
	-\Delta u-u\log|u|^2=\mu_0 u\ \   \text{in}\ \, \R^d,
\end{align}
which was first proposed by Bialynicki-Birula and Mycielski in 1976 (cf. \cite{lognls}) to describe the separability of non-interacting subsystems in quantum mechanics.
Employing the uniqueness  and  radial symmetry of positive solutions for  \eqref{nls} and \eqref{lognls1}, it was  proved successfully in \cite{wzq} that up to appropriate scalings, any ground state  of \eqref{nls} converges  to a ground state of the logarithmic  equation \eqref{lognls1} as \(p\searrow 1\). We  also refer to \cite{regu1,1983NA,wzqref26, uniq2014,zfl,normal} and the references therein for   more qualitative results of the logarithmic   equation \eqref{lognls1}.

In this paper, we are interested in the limiting behavior of ground states for the two-component (i.e., $N=2$) fermionic system (\ref{ep}) as $p\searrow 1$.  For this reason, we need to introduce the following logarithmic variational problem
\begin{align}\label{e}
	e^*_2:=&\inf\limits_{\gamma\in \mathcal{P}_2}\mathcal{E}(\gamma),
\end{align}
where $\mathcal{P}_2$ is as in \eqref{pl} with $N=2$, and
\begin{align}\label{func}
	\mathcal{E}(\gamma):=\text{Tr}(-\Delta\gamma)-\inte\rho_\gamma\text{log}\rho_\gamma dx,\ \ \ d\geq1.
\end{align}
Although there exists $\varphi\in H^1(\R^d, \C)$ such that $\inte|\varphi|^2\text{log}|\varphi|^2dx=-\infty$, see \cite{zfl} and the references therein,  it follows from the Lieb--Thirring inequality (cf. \cite[Equation (11)]{ii}) that  the energy functional $\mathcal{E}: \mathcal{P}_2\to\R\cup \{+\infty\}$ is well-defined, see also \eqref{3.29a} below.  Moreover, although \(\mathcal{E}\) is not differentiable on  \(\mathcal{P}_2\),  one can observe that any ground state of the fermionic system \eqref{e}  formally satisfies, up to a unitary transformation, the following two-component logarithmic Schr\"odinger system
\begin{align}\label{1.1}
	\arraycolsep=1.5pt
	\left\{\begin{array}{lll}
		-\Delta u_i-\text{log}\big(|u_1|^2+|u_2|^2\big) u_i=\mu_i u_i\ \
		\mathrm{in}\ \, \R^d,\\[3mm]
		\langle u_i,u_j\rangle_{L^2}=\delta_{ij},\ \  i, j=1,2,
	\end{array}\right.
\end{align}
where   all linear terms are combined together.

We shall prove rigorously in this paper that up to  translations and scalings,   ground states of the logarithmic fermionic  system  $e_2^*$ are the $L^\infty $-limits of   ground states for the  power-law  fermionic  system \eqref{ep} as $p\searrow 1$, where $N=2$. On the other hand, we also analyze  that  if $\gamma=\sum_{i=1}^2|u_i\rangle\langle u_i|$ is a minimizer of  $e_2^*$, then  up to a unitary transformation,  $u=(u_1,\,  u_2)$ satisfies the logarithmic  Schr\"odinger system \eqref{1.1}, where $\mu_1<\mu_2$ are the first two eigenvalues (counted with the multiplicity) of the Dirichlet realization for the operator
$$-\Delta-\log\big(|u_1|^2+|u_2|^2\big) \ \ \mbox{in}\,\ L^2(\Omega_u,\C).$$
Here
$\Omega_u:=\{x\in\R^d:\ |u_1(x)|^2+|u_2(x)|^2>0\}$, and the realization is defined in the quadratic-form sense of \eqref{op0} below.

We also recall from \cite{depino, wzq, 1987log} the following logarithmic Sobolev inequality:
\begin{align}\label{logsobo1}
	\inte |u|^2\text{log}|u|^2dx\leq\frac{d}{2}\text{log}\Big[\frac{2}{\pi d e}\inte|\nabla u|^2dx\Big],\ \   u\in H^1(\R^d, \C)\ \, \mbox{and}\,\ \|u\|_2^2=1,
\end{align}
where the  identity holds, if and only if $u(x)=e^{i\theta}\left(\frac{a}{\pi}\right)^{d/4}e^{-\frac{a}{2}|x-x_0|^2}$ holds for some $a>0$, $x_0\in\mathbb R^d$ and $\theta\in\mathbb R$. We also comment that Gross established in \cite{Gross}  a logarithmic Sobolev inequality for fermions in the noncommutative set of Clifford algebras, based on which Carlen and Lieb in \cite{CarlenLieb}  further analyzed its optimal form. Moreover,  Dolbeault et. al.  in \cite[Corollary 18]{DFLP} proved an optimal Euclidean logarithmic Sobolev inequality for orthonormal systems, for which the existence of optimizers was not addressed yet. By investigating the logarithmic variational problem $e_2^* $, in this paper we shall  establish a sharp logarithmic Sobolev inequality for orthonormal functions in \(\mathbb{R}^d\), distinct from those in \cite{Gross, CarlenLieb, DFLP}, for which we are able to prove the existence of optimizers.


\subsection{Main results}
The purpose of this subsection is to introduce main results of the present paper, all of which are also true essentially in the real-valued range. For simplicity,  we first give the following definition for the logarithmic  Schr\"odinger  system \eqref{1.1}.

\begin{defn}\label{dfn:1} (\emph{Ground states})  A vector   $U=(u_1, u_2 )\in H^1(\R^d,\C )\times H^1(\R^d,\C )$ is called a ground state of the logarithmic  Schr\"odinger  system \eqref{1.1}, if it solves \eqref{1.1} and $\mu_1<\mu_2$ are the first two eigenvalues (counted with multiplicity) of the Dirichlet realization for the operator
$$-\Delta-\log\big(|u_1|^2+|u_2|^2\big)\ \ \mbox{in}\ \, L^2(\Omega_u,\C),$$
where $\Omega_u:=\{x\in\R^d:\,|u_1(x)|^2+|u_2(x)|^2>0\}$, and the realization is understood in the   sense of \eqref{op0} below.
\end{defn}

Our first main result of the present paper is concerned with the following strong connection between  ground states of the logarithmic fermionic system \eqref{e} and ground states of the power-law fermionic  system \eqref{ep} as $p\searrow 1$, where $N=2$.

\begin{thm}\label{thm1}
Let $E_2(p)$ and $e_2^*$ be defined by \eqref{ep} and \eqref{e}, respectively.  Then we have
\begin{enumerate}
\item [(1).]  The ground state energy $E_2(p)$ satisfies
\begin{align}\label{1.5}
\lim\limits_{p\searrow1}(p-1)^{\frac{-2}{2-d(p-1)}}\Big[E_2(p)+\frac{2}{p}(p-1)^{\frac{d(p-1)}{2-d(p-1)}}\Big]=e_2^*>0.
\end{align}

\item [(2).] Let $\sum_{i=1}^2|v_{ip_n}\rangle\langle v_{ip_n}|$ be a minimizer of $E_2(p_n)$ satisfying \eqref{eqv}, where $p_n\searrow1$ as $n\to\infty $. Then  up to a subsequence and a translation if necessary,
\begin{equation}\label{1.9}
\begin{split}
&(p_n-1)^{\frac{-d}{4-2d(p_n-1)}}v_{ip_n}\big((p_n-1)^{\frac{-1}{2-d(p_n-1)}}x\big)\to u_i(x)\\
& \  \mbox{strongly in}\,\ H^1(\R^d, \C)\cap L^\infty(\R^d, \C)\,\  \mbox{as}\,\ n\to\infty,
\end{split}\end{equation}
where
 $\sum_{i=1}^2|u_i\rangle\langle u_i|$  is a minimizer of  the problem  $e_2^*$.
\end{enumerate}
\end{thm}


Theorem \ref{thm1} in particular gives the existence of ground states
for  the   logarithmic  fermionic system $e_2^*$. We remark that a similar result to  \eqref{1.9} was  previously established in \cite{zfl,wzq} for the convergence from the power-law equation \eqref{nls} to the logarithmic equation \eqref{lognls1}. We also note that   the analysis in \cite{wzq} relies crucially on the uniqueness and radial symmetry of positive solutions for \eqref{nls} and \eqref{lognls1}, whereas the scalar problem in \cite{zfl} involves a single linear coefficient that can be normalized to one, allowing the rescaled problem to be analyzed via the corresponding Nehari manifold.  Unfortunately,   the existing approaches of  \cite{zfl,wzq} are not applicable to the proof of Theorem \ref{thm1}. More precisely, one cannot generally expect to eliminate the two distinct Lagrange multipliers of \eqref{1.1} through a suitable scaling transformation.  Meanwhile, the uniqueness and radial symmetry of ground states for both \eqref{ep} and \eqref{e} are still  challenging open questions, due to their orthonormal constraints.

In order to overcome above difficulties, we shall carry out the proof of Theorem \ref{thm1} by  three steps.
As the first step of proving Theorem \ref{thm1}, we shall establish the following strict binding inequality
\begin{equation}\label{1.12M}
e_2^*<2e_1^*,
\end{equation}
 where $e_1^*$ denotes the ground-state energy of the
one-component logarithmic problem defined in \eqref{2.4b}. Towards this purpose,  through constructing a trial density matrix $\gamma_R\in\mathcal P_2$, we shall prove that the negative logarithmic correction dominates the positive kinetic energy correction  in the energy difference $\mathcal E(\gamma_R)-2e_1^*$ as $R\to\infty$. 
In the second step of proving Theorem \ref{thm1},  since two Lagrange multipliers of (\ref{eqv}) cannot in general be normalized simultaneously by a single scaling, we shall derive in Lemma \ref{lem2.1}  the refined estimate of the rescaled kinetic energy, based on which we then prove a uniform bound of Lagrange multipliers.  As the last step of proving Theorem \ref{thm1}, we employ the geometric localization method of \cite{ge}, together with the strict binding inequality (\ref{1.12M}), to exclude the case where the two-particle system splits into two spatially separated one-particle clusters. This helps us finally establish the $H^1$ convergence  \eqref{1.9}.

Our second main result of the present paper is to investigate the following qualitative properties of the  minimizers for $e_2^*$ defined by \eqref{e}.

\begin{thm}\label{thm2}
	Let $\big\{\gamma_n=\sum_{i=1}^2|u_{in}\rangle\langle u_{in}|\big\}$ be a minimizing sequence  of $e_2^*$, and suppose    $\gamma=\sum_{i=1}^2|u_i\rangle\langle u_i|$ is a minimizer of $e_2^*$. Then we have
	\begin{enumerate}
		\item [(1).] The sequence $\{(u_{1n}, u_{2n})\}$ is, up to  translations, relatively compact in $H^1(\R^d,\C)\times H^1(\R^d,\C)$. 	
		
		\item [(2).]  Up to a unitary transformation,  $u=(u_1, u_2 )\in C^2(\R^d, \C)\times C^2(\R^d, \C)$ is a ground state of  the logarithmic  Schr\"odinger  system  \eqref{1.1} and satisfies the  Gaussian decay
		$$
		|u_1(x)|+|u_2(x)|\leq C e^{-\theta|x|^2}\ \ \text{in}\ \, \R^d
		$$
		for some constants $\theta>0$ and $C>0$.
	\end{enumerate}
\end{thm}

The primary challenging point of proving Theorem \ref{thm2} (1) is to rule out the vanishing of the minimizing sequence for \( e_2^* \). To address this difficulty, we employ the asymptotic formula  \eqref{1.5} to relate the logarithmic term of (\ref{func}) to the limit of the renormalized power nonlinearity in (\ref{ep}) as $p\searrow 1$.  On the other hand, in order to prove Theorem \ref{thm2} (2), since the potential $-\log(|u_1|^2+|u_2|^2) $ may be singular on the boundary of $\Omega_\gamma:=\{x\in\mathbb R^d:\rho_\gamma(x)>0\},$ which may be disconnected,  it is generally difficult to prove that  $u=(u_1, u_2)$ is a ground state of  the logarithmic  Schr\"odinger  system \eqref{1.1}. To handle this difficulty,  we shall define  the corresponding Dirichlet operator $H_\gamma^D$ in (\ref{op0}), from which we then prove that the operator  $H_\gamma^D$ has a compact resolvent. The min-max  argument thus yields that the corresponding two Lagrange multipliers  $\mu_1\leq\mu_2$ are precisely the first two eigenvalues of  $H_\gamma^D$, counted with the multiplicity.
Finally, without imposing any connectedness assumption on \(\Omega_\gamma\), we shall prove the simplicity of the first eigenvalue $\mu_1$ by using the strict binding inequality (\ref{1.12M}).

As a byproduct of Theorem 1.1, we next relate  the minimization problem $e_2^*$ to the following  logarithmic Sobolev inequality for orthonormal functions.

\begin{thm}\label{thmlog}
Suppose that $\{u_i\}_{i\geq1}\subset H^1(\mathbb R^d,\mathbb C)$
is an orthonormal system in $L^2(\mathbb R^d,\mathbb C)$,
and let $\{n_i\}_{i\geq1}\subset[0,1]$ satisfy $\sum_{i\geq1}n_i=2$ and  $\sum_{i\geq1}n_i\int_{\mathbb R^d}|\nabla u_i|^2\,dx<\infty.$
Then
\begin{align}\label{logsobo}
\int_{\R^d}\Big(\sum\limits_{i\geq1}n_i|u_i|^2\Big)\log\Big(\sum\limits_{i\geq1}n_i|u_i|^2\Big)dx
\leq d\log\bigg[ \frac{e^{1-\frac{e_2^*}{d}}}{d}\sum\limits_{i\geq1}n_i\inte |\nabla u_i|^2dx
\bigg],
\end{align}
where the  best constant can be  attained, and the constant $e_2^*>0$ is as in \eqref{e}.
\end{thm}

\begin{rem} Theorem \ref{thmlog} provides a sharp logarithmic Sobolev inequality for orthonormal functions, which can be regarded  as an orthogonal generalization of the classical Euclidean logarithmic Sobolev inequality \eqref{logsobo1}, which was studied  in \cite{depino, wzq, 1987log} and the references therein.
On the other hand,  the proof  of Theorem \ref{thmlog} essentially  yields  (cf. \eqref{M:3.42} below) the following family of logarithmic Sobolev
inequalities:  for  any $ \lambda>0$,
\begin{equation}\label{1:inequality}
\int_{\mathbb R^d}\rho_\gamma\log\rho_\gamma\,dx\leq \lambda^2\Tr(-\Delta\gamma)-\big(e_2^*+2d\log\lambda\big), \ \, \forall\, \gamma\in \mathcal{K}_2,
\end{equation}
where  the identity holds if $\lambda=1$ and $\gamma$ is a minimizer of $e_2^*$, and the set $\mathcal{K}_2$ is defined by
\begin{align*}
	\mathcal{K}_2
	:=
	\big\{
	\gamma:\ 
	0\leq \gamma=\gamma^*\leq 1,\
	\Tr(\gamma)=2,\
	\Tr(-\Delta\gamma)<\infty
	\big\}.
\end{align*}
We remark that the inequality (\ref{1:inequality}) can be  thought of as an orthogonal generalization of the   logarithmic Sobolev inequality analyzed in \cite{analysis, wzq} and the references therein.
\end{rem}

As addressed in Remark \ref{rem} below, we comment that every minimizer of $e_2^*$ is an optimizer of \eqref{logsobo}, and  up to a scaling,   every optimizer  of \eqref{logsobo} is also a minimizer of the following relaxed problem 
\begin{align}\label{1.14}
	e_2^*=\inf\limits_{\gamma\in\mathcal{K}_2}\mathcal{E}(\gamma).
\end{align}
Actually,  Theorem \ref{thmlog} is a direct consequence of the equivalence \eqref{1.14}. Moreover, since the logarithmic term $\inte\rho\log\rho dx$ is not continuous  and exhibits a singularity as $\rho\to0^+$,
the  finite-rank approximation  argument of \cite{i} used for the power-law problem (\ref{ep}) does not directly apply to the logarithmic nonlinear problem (\ref{1.14}). To overcome this difficulty in proving Theorem \ref{thmlog}, we shall construct a measurable decomposition $\{\gamma_t\}_{t\in[0, 1)}$  of every admissible mixed state $\gamma\in\mathcal{K}_2$ into rank-two orthogonal  projections, from which we finally prove that the average energy of the resulting projections does not exceed the energy of  \(\gamma\), in the sense that $\mathcal{E}(\gamma)\geq \int_0^1 \mathcal{E}(\gamma_t)dt$.

This paper is organized as follows. In Section \ref{sec2}, we shall derive some estimates for $E_2(p)$ and its minimizers as $p\searrow1$. In Section \ref{section3}, we then analyze the convergence of minimizers for $E_2(p)$ as $p\searrow1$, based on which Subsection 3.1 is devoted to the proof of Theorem \ref{thm1}. The proofs of Theorems  \ref{thm2} and \ref{thmlog}  are addressed in Section 4, which is concerned with the qualitative properties of the logarithmic  fermionic problems $e_2^*$ and the  logarithmic Sobolev inequality (\ref{logsobo}). 



\section{Some Estimates of $E_2(p)$  as $p\searrow1$}\label{sec2}

The purpose of this section is to establish some analytical estimates, which are needed to study the convergence of the energy $E_2(p)$ and its minimizers as $p\searrow1$,  for which we always suppose that $d\geq 1$ and $p\in \big(1,\, \min(2, 1+2/d)\big)$.

Let $E_2(p)$  and  $e_2^*$ be defined by \eqref{ep} and \eqref{e}, respectively. Define
\begin{equation}\label{problem}
e_N(p):=\inf\limits_{\gamma\in\mathcal{P}_N}\mathcal{E}_p(\gamma), \ \, N=1, 2,
\end{equation}
where  $\mathcal{P}_N$ is as in \eqref{pl},  and
$$
\mathcal{E}_p(\gamma):=\text{Tr}(-\Delta\gamma)-\frac{1}{p(p-1)}\inte\big(\rho_\gamma^p-\rho_\gamma\big)dx.
$$
Note that for any $\tilde{\gamma}\in\mathcal{P}_N$,
\begin{align}\label{2.4a}
\text{Tr}(-\Delta\tilde{\gamma})-\frac{1}{p}\inte\rho_{\tilde{\gamma}}^pdx
=&\alpha_p^2\text{Tr}(-\Delta\gamma)-\frac{\alpha_p^2}{p(p-1)}\inte(\rho_{\gamma}^p-\rho_{\gamma})dx-\frac{N\alpha_p^2}{p(p-1)}\\
=&\alpha_p^2\Big[\mathcal{E}_p(\gamma)-\frac{N}{p(p-1)}\Big],\nonumber
\end{align}
where $\gamma(x,y):=\alpha_p^{-d}\tilde{\gamma}\big(\alpha_p^{-1}x,\ \alpha_p^{-1}y\big)$ and  $ \alpha_p:=(p-1)^{\frac{1}{2-d(p-1)}}$.
This yields that
\begin{align}\label{2.4}
e_N(p)=
(p-1)^{\frac{-2}{2-d(p-1)}}\Big[E_N(p)+\frac{N}{p}(p-1)^{\frac{d(p-1)}{2-d(p-1)}}\Big],
\end{align}
where the problem $E_N(p)$ is given by \eqref{ep}.


Let $e_N^*$ and $e_N(p)$ be defined by \eqref{e} and \eqref{problem}, respectively,
where $N=1,2$, and
\begin{align}\label{2.4b}
e_1^*:=\inf\Big\{\inte |\nabla u|^2dx-\inte |u|^2\text{log}|u|^2dx:\, u\in H^1(\R^d,\C),\ \|u\|_2^2=1\Big\}.
\end{align}
The following lemma presents the energy relationship between
$e_N^*$ and $e_N(p)$ as $p\searrow1$.


\begin{lem} \label{lemstri}
The energy $e_2(p)$ satisfies
$$
0<\liminf\limits_{p\searrow1}e_2(p)\leq\limsup\limits_{p\searrow1}e_2(p)\leq e_2^*<2e_1^*=2\lim\limits_{p\searrow1}e_1(p)=2d+d \log\, \pi.
$$
\end{lem}

\noindent \textbf{Proof.}  We first prove that
\begin{align}\label{2.20}
	\liminf\limits_{p\searrow1}e_2(p)>0.
\end{align}
It follows from \cite[Equation (34)]{i} that
\begin{align}\label{2.5b}
E_2(p)\geq-\frac{2-d(p-1)}{p}\Big(\frac{d(p-1)}{2pc_{LT}(d)}\Big)^{\frac{d(p-1)}{2-d(p-1)}},
\end{align}
where
\begin{align}\label{lt}
0<c_{LT}(d):=\inf\Big\{\frac{\|\gamma\|^{2/d}\text{Tr}(-\Delta\gamma)}{\inte\rho_\gamma^{1+2/d}dx}:\  0\leq \gamma=\gamma^*\Big\}.
\end{align}
We then calculate from \eqref{2.4} and \eqref{2.5b} that
\begin{align}\label{2.6}
e_2(p)=&
(p-1)^{\frac{-2}{2-d(p-1)}}\Big[E_2(p)+\frac{2}{p}(p-1)^{\frac{d(p-1)}{2-d(p-1)}}\Big]\nonumber\\
\geq&(p-1)^{\frac{-2}{2-d(p-1)}}\Big[-2\frac{2-d(p-1)}{2p}\Big(\frac{d(p-1)}{2pc_{LT}(d)}\Big)^{\frac{d(p-1)}{2-d(p-1)}}+\frac{2}{p}(p-1)^{\frac{d(p-1)}{2-d(p-1)}}\Big]\nonumber\\
=&\frac{2}{p(p-1)}\Big[1-\frac{2-d(p-1)}{2}\Big(\frac{d}{2pc_{LT}(d)}\Big)^{\frac{d(p-1)}{2-d(p-1)}}\Big]\\
=&\frac{2}{p(p-1)}\bigg\{1-\Big[1-\frac{d(p-1)}{2}\Big]\Big[1+\frac{d(p-1)}{2-d(p-1)}\text{log}\frac{d}{2pc_{LT}(d)}+o(p-1)\Big]\bigg\}\nonumber\\
=&d\Big(1-\text{log}\frac{d}{2c_{LT}(d)}\Big)+o(1)\ \ \ \text{as}\ \ p\searrow1.\nonumber
\end{align}
Note from \cite{clt, open} that $c_{LT}(d)\geq \pi^{1-\frac{2}{d}}\frac{4d}{d+2}\big(d\Gamma(d/2)\big)^{2/d}
$,  
where $\Gamma(\cdot)$ denotes   $Gamma$ function.
By direct calculations, we hence conclude from  \eqref{2.6} that
\begin{align}\label{2.8c}
\liminf\limits_{p\searrow1}e_2(p)>0,\ \   \forall\ 1\leq d\leq8.
\end{align}

As for the case  $d>8$,  we claim that
\begin{align}\label{2.8b}
c_{LT}(d)/d\geq\frac{2\big(\Gamma(4)\big)^{1/4}}{5\pi},\ \ \ \forall\ d>8.
\end{align}
Actually, set $$f(t):=\frac{1}{t}\text{log}\Gamma(t)-\text{log}(2t+2),$$
and  note that
\begin{equation}\label{2.0}
\begin{split}\frac{c_{LT}(d)}{d}
\geq4\pi^{1-\frac{2}{d}}\
\frac{\big(d\Gamma(d/2)\big)^{2/d}}{d+2}
\geq4\pi^{-1}\
\frac{\big(\Gamma(d/2)\big)^{2/d}}{d+2}=4\pi^{-1}e^{f(d/2)},\ \ \ \forall\ d\geq1.
\end{split}\end{equation}
Since $$\big(\text{log}\Gamma(t)\big)''=\int_0^\infty\frac{xe^{-tx}}{1-e^{-x}}dx\ \  \mbox{in}\,\ \R^+,$$
and
\begin{equation}\label{a.1}
\begin{split}
f'(t)=\frac{1}{t^2}\Big[t\big(\text{log}\Gamma(t)\big)'-\text{log}\Gamma(t)-\frac{t^2}{t+1}\Big]:=\frac{h(t)}{t^2}\ \  \text{in}\ \, \R^+,
\end{split}
\end{equation}
we deduce that
\begin{equation}\label{a.2}
\begin{split}
h'(t)=&t\big(\text{log}\Gamma(t)\big)''-\frac{t^2+2t}{(1+t)^2}\geq\int_0^\infty\frac{txe^{-tx}}{1-e^{-x}}dx-1\\
=&\frac{1}{t}\int_0^\infty\frac{ye^{-y}}{1-e^{-y/t}}dy-1\geq\int_0^\infty e^{-y}dy-1=0\ \   \text{in}\ \, \R^+.
\end{split}
\end{equation}
One can   check  that $h(4)>0$, which then implies from \eqref{a.1} and \eqref{a.2} that $f'(t)>0$ holds for any $t\geq4$. 
We thus obtain from \eqref{2.0} that the claim \eqref{2.8b} holds true.  The claim \eqref{2.20} therefore follows from  \eqref{2.6}--\eqref{2.8b}.

Moreover, since
\begin{align}\label{2.8a}
\frac{t^{p-1}-1}{p-1}=\frac{e^{(p-1)\text{log}t}-1}{(p-1)\text{log}t}\text{log}t\geq\text{log}t,\ \   \forall\ t>0,
\end{align}
 we have
\begin{align*}
\limsup\limits_{p\searrow1}e_2(p)\leq\limsup\limits_{p\searrow1}\mathcal{E}_p(\gamma)
=&\limsup\limits_{p\searrow1}\Big[\text{Tr}(-\Delta\gamma)-\frac{1}{p(p-1)}\inte\big(\rho_\gamma^p-\rho_\gamma\big)dx\Big]\\
\leq&\limsup\limits_{p\searrow1}\Big[\text{Tr}(-\Delta\gamma)-\frac{1}{p}\inte \rho_{\gamma}\text{log}\rho_{\gamma}dx\Big]\\
=&\text{Tr}(-\Delta\gamma)-\inte \rho_{ \gamma}\text{log}\rho_{\gamma}dx,\ \   \forall\ \gamma\in \mathcal{P}_2.
\end{align*}
This further implies that $\limsup\limits_{p\searrow1}e_2(p)\leq e_2^*$.
Note from \cite{normal,lognls} that $$\lim\limits_{p\searrow1}e_1(p)=e_1^*=d+\frac{d}{2}\text{log}\pi,$$
and  the problem $e_1^*$ given by \eqref{2.4b} admits, up to translations, a unique positive minimizer $w(x)=\pi^{-\frac{d}{4}}e^{-\frac{|x|^2}{2}}$. Thus, it remains to show that  $e_2^*<2e_1^*$.

Denote
\begin{equation}\label{2.15b}
\begin{split}
G_R=:
\left(
\begin{split}
&w\\[-2mm]
&w_R
\end{split}
\right)
\big(w, w_R\big)
=\begin{bmatrix}
1& a_{R} \\
a_{R}& 1
\end{bmatrix}
,\ \   a_{R}:=\langle w,\,  w_{R}\rangle,
\end{split}
\end{equation}
where $w_R(x):=w(x-Re_1)$, $R>0$ and $e_1:=(1,0,\cdots,0)$. One can check that
the matrix $G_R$ is positive definite for sufficiently large $R>0$. Hence, the components of the vector
\begin{equation}\label{2.15a}
(u_{1R},\,  u_{2R}):=(w, \, w_{R})G_R^{-\frac{1}{2}}
\end{equation}
is orthonormal in $L^2(\R^d)$  for sufficiently large $R>0$. Define for sufficiently large $R>0$,
\begin{equation}\label{2.15c}
 \gamma_R:=|u_{1R}\rangle\langle u_{1R}|+|u_{2R}\rangle\langle u_{2R}|\in\mathcal{P}_2.
\end{equation}
Similar to \cite{i}, it then gives that
\begin{equation}\label{8}
\begin{split}
\gamma_R
=&|w\rangle\langle w|+|w_{R}\rangle\langle w_{R}|-a_R|w\rangle\langle w_{R}|-a_R|w_{R}\rangle\langle w|\\
&+O(a_R^2)\Big[|w\rangle\langle w|+|w_{R}\rangle\langle w_{R}|+|w\rangle\langle w_{R}|+|w_{R}\rangle\langle w|\Big]
\ \ \ \text{as}\ \ R\to\infty,
\end{split}
\end{equation}
where $g(R)=O(a_R^2)$ means that there exists a constant $C > 0$, independent of $R > 0$, such that $|g(R)| \leq Ca_R^2$  holds for any sufficiently large $R > 0$.

Since $w(x)=\pi^{-\frac{d}{4}}e^{-\frac{|x|^2}{2}}$,  direct calculations yield that
\begin{equation}\label{2.24a}
\begin{split}
a_R=\inte ww_Rdx=e^{-\frac{R^2}{4}},\ \   \inte ww_R\text{log}w^2dx=-e^{-\frac{R^2}{4}}\Big[\frac{R^2}{4}+\frac{d}{2}(1+\text{log}\pi)\Big],
\end{split}
\end{equation}
and thus
\begin{equation}\label{2.25b}
\begin{split}
\mathrm{Tr}(-\Delta\gamma_R)
=&2\inte |\nabla w|^2dx
-2a_{R}\inte \nabla w\cdotp\nabla w_{R}dx+O(a_R^2)\\
=&2\inte |\nabla w|^2dx
-2a_{R}\Big[e_1^*a_R+\inte ww_R\text{log}w^2dx\Big]
+O(a_R^2)\\
=&2\inte |\nabla w|^2dx+\frac{R^2}{2}e^{-\frac{R^2}{2}}+O(e^{-\frac{R^2}{2}})\ \   \mathrm{as}\ \, R\to\infty,
\end{split}
\end{equation}
where the second identity follows from the fact that
\begin{equation*}
-\Delta w-w\text{log}w^2=e_1^* w\ \  \, \text{in}\ \, \R^d.
\end{equation*}
Moreover, since  $\text{log}(1+t)\leq t$ holds for any $t>0$, one can calculate from \eqref{2.24a} that
\begin{equation*}
\begin{split}
&\inte ww_R\text{log}(w^2+w_R^2)dx\\
=&\inte ww_R\text{log}w^2dx+\inte ww_R\text{log}(1+w_R^2w^{-2})dx\\
=&\inte ww_R\text{log}w^2dx+\inte ww_R\text{log}(1+e^{|x|^2-|x-Re_1|^2})dx\\
\leq&\inte ww_R\text{log}w^2dx+\inte |x|^2ww_Rdx+(2\text{log}2)\inte ww_Rdx\\
=&O\big(e^{-\frac{R^2}{4}}\big)\ \   \text{as}\ \, R\to\infty,
\end{split}
\end{equation*}
and
\begin{align*}
\infty>&2\inte w^2\text{log}w^2dx+2\inte w^2dx\nonumber\\
\geq&\inte(w^2+w_R^2)\text{log}(w^2+w_R^2)dx\nonumber\\
=&2\inte w^2\text{log}w^2dx+2\inte w^2\text{log}\big(1+w_R^2w^{-2}\big)dx\\
\geq &2\inte w^2\text{log}w^2dx+2\pi^{-\frac{d}{2}}\int_{B_1\big(\frac{Re_1}{2}+2e_1\big)} e^{-(\frac{R}{2}+3)^2}\text{log}\big(1+e^{2R}\big)dx\nonumber\\
\geq&2\inte w^2\text{log}w^2dx+ 4\pi^{-\frac{d}{2}}|B_1|Re^{-(\frac{R}{2}+3)^2},\ \   \forall\,\ R>10.\nonumber
\end{align*}
Since  it follows from \eqref{2.15b}--\eqref{2.15c} that for any sufficiently large $R > 0$,
\[
\begin{aligned}
\rho_{\gamma_R}=u_{1R}^2+u_{2R}^2=(w,w_R)G_R^{-1}
\begin{pmatrix}w\\w_R
\end{pmatrix}
=\frac{w^2+w_R^2-2a_Rww_R}{1-a_R^2},
\end{aligned}
\]
applying  the convexity of $t\log t$  on $(0,\infty)$, we obtain   that
\begin{align}\label{2.27a}
	&\inte\rho_{\gamma_R}\log\rho_{\gamma_R}\,dx\nonumber\\
	\geq&
	\inte(w^2+w_R^2)\log(w^2+w_R^2)\,dx+\inte\big[\rho_{\gamma_R}-(w^2+w_R^2)\big]
	\big[1+\log(w^2+w_R^2)\big]\,dx\nonumber\\
	=&
	\inte(w^2+w_R^2)\log(w^2+w_R^2)\,dx\\
	&+\frac{a_R^2}{1-a_R^2}
	\inte(w^2+w_R^2)\log(w^2+w_R^2)\,dx-\frac{2a_R}{1-a_R^2}
	\inte ww_R\log(w^2+w_R^2)\,dx\nonumber\\
	\geq{}&
	\inte(w^2+w_R^2)\log(w^2+w_R^2)\,dx-O(a_R^2)\nonumber\\
	\geq{}&
	2\inte w^2\log w^2\,dx
	+4\pi^{-\frac d2}|B_1|R e^{-(\frac R2+3)^2}
	-O\big(e^{-\frac{R^2}{2}}\big)\ \ \text{as}\ \ R\to\infty.\nonumber
\end{align}
As a consequence of \eqref{2.25b} and \eqref{2.27a},  we conclude that
\begin{align*}\label{substri}
e_2^*\leq& \mathrm{Tr}(-\Delta\gamma_R)-\inte\rho_{\gamma_R}\text{log}\rho_{\gamma_R}dx\nonumber\\
\leq&2e_1^*+R^2e^{-\frac{R^2}{2}}- 4\pi^{-\frac{d}{2}}|B_1|Re^{-(\frac{R}{2}+3)^2}\\[1.5mm]
<&2e_1^*,\ \   \text{if}\ R>0 \  \text{is sufficiently  large},\nonumber
\end{align*}
which  therefore completes the proof of Lemma \ref{lemstri}.  \qed

\vspace{.15cm}

Let $E_2(p)$ be given by \eqref{ep}, and suppose   $\tilde{\gamma}_p=\sum_{i=1}^2|v_{ip}\rangle\langle v_{ip}|\in\mathcal{P}_2$ is a minimizer of $E_2(p)$  satisfying \eqref{eqv}.
Define
\begin{equation}\label{min}
\gamma_p:=\sum_{i=1}^2|u_{ip}\rangle\langle u_{ip}|,\ \   u_{ip}(x):=\alpha_p^{-\frac{d}{2}}v_{ip}(\alpha_p^{-1}x),\ \    \alpha_p:=(p-1)^{\frac{1}{2-d(p-1)}}.
\end{equation}
Applying \eqref{eqv}, it then follows from  \eqref{2.4a} and \eqref{2.4} that $\gamma_p$
is a minimizer of $e_2(p)$ given by \eqref{problem},  and  $u_{ip}$ satisfies
\begin{align}\label{equ}
-\Delta u_{ip}=&\frac{\rho_{\gamma_p}^{p-1}-1}{p-1}u_{ip}
+\frac{1+\mu_{ip}\alpha_p^{-d(p-1)}}{p-1}u_{ip}\ \ \ \text{in}\ \, \R^d,\ \,   i=1, 2.
\end{align}
We shall prove that up to a subsequence and  a translation,  the limit operator $\gamma$ of the above  sequence $\{\gamma_p\}$ as  $p\searrow1$ is essentially a minimizer of the problem  $e_2^*$. Towards this aim, we first establish the following uniform estimates of $\{\gamma_p\}$ and  $\{(\mu_{1p}, \mu_{2p})\}$.

\begin{lem}\label{lem2.1}
Let  $\gamma_p$ and  $\mu_{1p}<\mu_{2p}<0$ be given by \eqref{min} and \eqref{equ}, respectively. Then we have
\begin{align}\label{2.15}
\lim\limits_{p\searrow1}\Tr(-\Delta\gamma_p)=d,
\ \ \,
\sup\limits_{p\searrow1}\sum_{i=1}^2\frac{\big|1+\mu_{ip}\alpha_p^{-d(p-1)}\big|}{p-1}<\infty, 
\end{align}
where $\alpha_p=(p-1)^{\frac{1}{2-d(p-1)}}>0$ is as in \eqref{min}.
\end{lem}

\noindent \textbf{Proof.}
We first prove the claim that
\begin{align}\label{2.18}
\lim\limits_{p\searrow1}\text{Tr}(-\Delta\gamma_p)=d.
\end{align}
Let $v_p>0$ be a minimizer of $E_1(p)$, and suppose   $\varphi_p>0$ is the unique (up to translations) positive  minimizer of the following  variational problem
$$
0<\sigma_p:=\inf\limits_{\|\varphi\|^{2p}_{2p}=1}\inte\big(|\nabla \varphi|^2+\varphi^2\big)dx.
$$
It then follows from \cite[Equation (44)]{i} and \cite[Theorem 1.2]{wzq} that
\begin{equation}\label{2.3}
-\Delta v_p-|v_p|^{2p-2}v_p=\mu_pv_p\ \ \, \text{in}\ \, \R^d, \ \ \,  \mu_p=E_1(p)\frac{2p-d(p-1)}{2-d(p-1)}<0,
\end{equation}
\begin{equation}\label{2.3c}
-\Delta\varphi_p+\varphi_p-\sigma_p|\varphi_p|^{2p-2}\varphi_p=0\ \   \text{in}\ \, \R^d,
\end{equation}
and
\begin{equation}\label{2.19}
\lim\limits_{p\searrow1}\Big[(2p-2)^{\frac{d(p-1)}{2p}-1}\sigma_p-(2p-2)^{-1}\Big]=\frac{d}{2}\Big(1+\frac{\text{log}(2\pi)}{2}\Big)\ \ \ \text{as}\ \ p\searrow1.
\end{equation}
By the uniqueness of  positive solutions for \eqref{2.3c},  one can  verify  from \eqref{2.3} that 
\begin{equation}\label{2.20a}
 E_1(p)=
 -\sigma_p^{-\frac{2p}{2-d(p-1)}}\frac{2-d(p-1)}{2p-d(p-1)}
 \left(\frac{2p}{2p-d(p-1)}\right)^{\frac{2(p-1)}{2-d(p-1)}},
\end{equation}
which further yields   from  \cite[Lemma 12]{i}  that
\begin{equation}\label{2.3a}
	E_2(p)\leq 2E_1(p)=-2\sigma_p^{-\frac{2p}{2-d(p-1)}}\frac{2-d(p-1)}{2p-d(p-1)}
	\left(\frac{2p}{2p-d(p-1)}\right)^{\frac{2(p-1)}{2-d(p-1)}}.
\end{equation}
We hence conclude from \eqref{2.19} and \eqref{2.3a}  that
\begin{equation}\label{2.2a}
\limsup\limits_{p\searrow1}E_2(p)\leq-2.
\end{equation}

On the other hand, recall from \cite{i} that
\begin{equation*}
\begin{split}
&\frac{d(p-1)-2}{d(p-1)}\text{Tr}(-\Delta\tilde{\gamma}_p)=\frac{d(p-1)-2}{2p}\inte\rho_{\tilde{\gamma}_p}^pdx\\
=&E_2(p)\geq-2\frac{2-d(p-1)}{2p}\Big(\frac{d(p-1)}{2pc_{LT}(d)}\Big)^{\frac{d(p-1)}{2-d(p-1)}},
\end{split}\end{equation*}
where   $\tilde{\gamma}_p:=\sum_{i=1}^2|v_{ip}\rangle\langle v_{ip}|$, and $v_{ip}$ and the constant $c_{LT}(d)>0$ are as in \eqref{min} and  \eqref{lt}, respectively.
We then obtain from \eqref{2.2a} that
\begin{equation*}
\lim\limits_{p\searrow1}E_2(p)=-2,\ \ \, \lim\limits_{p\searrow1}(p-1)^{-1}\text{Tr}(-\Delta\tilde{\gamma}_p)=d,\ \ \, \lim\limits_{p\searrow1}\inte\rho_{\tilde{\gamma}_p}^pdx= 2,
\end{equation*}
and thus
\begin{align*}
\lim\limits_{p\searrow1}\text{Tr}(-\Delta\gamma_p)=\lim\limits_{p\searrow1}(p-1)^{-1}\text{Tr}(-\Delta\tilde{\gamma}_p)=d.
\end{align*}
This   proves the claim \eqref{2.18}.

We next note from \cite[Lemma 2.1 (i)]{wzq} that for any $q>1$, there exists a constant $C(q)>0$, depending only on $q$, such that
\begin{align}\label{2.28a}
\frac{t^{p-1}-1}{p-1}\leq C(q)t^{q-1},\ \   \forall\, t\geq0,\,\ p\in(1, q).
\end{align}
Applying \eqref{lt}, we then get  from  \eqref{equ} and  \eqref{2.18} that
\begin{align}\label{2.11}
-\frac{2+\sum_{i=1}^2\mu_{ip}\alpha_p^{-d(p-1)}}{p-1}
\leq&\int_{\rho_{\gamma_p}(x)<1}\frac{1-\rho_{\gamma_p}^{p-1}}{p-1}\rho_{\gamma_p}dx-\frac{2+\sum_{i=1}^2\mu_{ip}\alpha_p^{-d(p-1)}}{p-1}\nonumber\\
=&-\text{Tr}(-\Delta\gamma_p)+\int_{\rho_{\gamma_p}(x)\geq1}\frac{\rho_{\gamma_p}^{p-1}-1}{p-1}\rho_{\gamma_p}dx\\
\leq &C\int_{\rho_{\gamma_p}(x)\geq1}\rho_{\gamma_p}^{1+\frac{2}{d}}dx\leq Cc^{-1}_{LT}(d) \sup\limits_{p\searrow1}\text{Tr}(-\Delta\gamma_p)
<\infty.\nonumber
\end{align}
Denote $\tilde{\gamma}'_p:=\tilde{\gamma}_p-|v_{2p}\rangle\langle v_{2p}|\in\mathcal{P}_1$. Applying the convexity of $t\mapsto t^p$, it can be verified  from \eqref{eqv}  that
\begin{align*}
E_1(p)\leq&\text{Tr}(-\Delta\tilde{\gamma}'_p)-\frac{1}{p}\inte\rho_{\tilde{\gamma}'_p}^pdx\nonumber\\
=&\text{Tr}(-\Delta\tilde{\gamma}_p)-\frac{1}{p}\inte\rho_{\tilde{\gamma}_p}^pdx-\mu_{2p}-\frac{1}{p}\inte \Big[
\rho_{\tilde{\gamma}'_p}^{p}
-\rho_{\tilde{\gamma}_p}^{p}
-p\rho_{\tilde{\gamma}_p}^{p-1}
\big(
\rho_{\tilde{\gamma}'_p}
-\rho_{\tilde{\gamma}_p}
\big)
\Big]dx\\
\leq&\text{Tr}(-\Delta\tilde{\gamma}_p)-\frac{1}{p}\inte\rho_{\tilde{\gamma}_p}^pdx-\mu_{2p}=E_2(p)-\mu_{2p},\nonumber
\end{align*}
which then yields that
\begin{align}\label{M:2.28a}\mu_{2p}\leq E_2(p)-E_1(p)\leq E_1(p).\end{align}

Set $\mu_*:=\frac{d}{2}(1+\frac{\text{log}(2\pi)}{2})+1>0$.
Applying (\ref{M:2.28a}), we then calculate from \eqref{2.19} and \eqref{2.20a} that
\begin{align*}
	&\mu_{1p}\alpha_p^{-d(p-1)}
	<\mu_{2p}\alpha_p^{-d(p-1)}
	\leq
	E_1(p)\alpha_p^{-d(p-1)}
	\\
	={}&
	-\frac{2-d(p-1)}{2p-d(p-1)}
	\left(\frac{2p}{2p-d(p-1)}\right)^{
		\frac{2(p-1)}{2-d(p-1)}}
	\sigma_p^{-\frac{2p}{2-d(p-1)}}
	(p-1)^{-\frac{d(p-1)}{2-d(p-1)}}
	\\
	\leq{}&
	-\frac{2-d(p-1)}{2p-d(p-1)}
	\left(\frac{2p}{2p-d(p-1)}\right)^{
		\frac{2(p-1)}{2-d(p-1)}}
	\bigl(1+2\mu_*(p-1)\bigr)^{
		-\frac{2p}{2-d(p-1)}}
	2^{\frac{d(p-1)}{2-d(p-1)}}
	\\
	=&
	-\left[1-\frac{2(p-1)}{2p-d(p-1)}\right]\, 
	\left[
	1-\frac{4p\mu_*(p-1)}{2-d(p-1)}
	\right]
	\left[
	1+\frac{d(p-1)\log2}{2-d(p-1)}
	\right]+o(p-1)
	\\
	=&
	-1+\left(2\mu_*+1-\frac d2\log2\right)(p-1)
	+o(p-1)\ \ \text{as}\ \ p\searrow1,
\end{align*}
which further yields that
\begin{align*}
\limsup\limits_{p\searrow1}\frac{1+\mu_{ip}\alpha_p^{-d(p-1)}}{p-1}
\leq-\frac{d\text{log}2}{2}+2\mu_*+1,\ \   i=1, 2.
\end{align*}
Thus, we immediately obtain from \eqref{2.11} that
\begin{align*}
\sup\limits_{p\searrow1}\frac{\big|1+\mu_{ip}\alpha_p^{-d(p-1)}\big|}{p-1}<\infty,\ \ \ i=1, 2.
\end{align*}
This therefore completes the proof of  Lemma \ref{lem2.1}.\qed

\section{Convergence of $E_2(p)$ as $p\searrow1$}\label{section3}

In this section, we first analyze the convergence of minimizers for $E_2(p)$ as $p\searrow1$, based on which we shall complete in Subsection 3.1  the proof of Theorem \ref{thm1}.

For simplicity, we define
\begin{equation}\label{f}
H(N):=\sum_{i=1}^N(-\Delta_{x_i}),\ \ \ F(\rho):=-\frac{1}{p(p-1)}\inte (\rho^p-\rho)dx.
\end{equation}
We also denote $H^1_a(\R^{dN},\C)$ as the subspace of $H^1(\R^{dN},\C)$, which consists of all  antisymmetric wave functions. Since the functional $F(\rho)$ is concave with respect to $\rho$ on the set $\{\rho\geq0\}$, the same argument of  \cite[Equation (1.7)]{cg1} gives that
\begin{equation}\label{eq}
e_N(p)=\inf\Big\{\big\langle \Psi,\, H(N)\Psi\big\rangle+F(\rho_\Psi):\, \Psi\in H^1_a(\R^{dN}, \C),\ \|\Psi\|_2^2=1\Big\},\ \ N=1,2,
\end{equation}
where the problem $e_N(p)$ is given by  \eqref{problem}, and
\begin{equation*}
\begin{split}
\rho_\Psi(x):=&N\int_{\R^{d(N-1)}}|\Psi(x,x_2,\cdots,x_N)|^2dx_2\cdots dx_N.
\end{split}
\end{equation*}
Applying Lemmas \ref{lemstri} and \ref{lem2.1}, we first establish the following $H^1$-uniform convergence.  

\begin{lem}\label{prop}
Suppose $\gamma_{p_n}=\sum_{i=1}^2|u_{ip_n}\rangle\langle u_{ip_n}|\in\mathcal{P}_2$  given by \eqref{min}  is a minimizer of $e_2(p_n)$, where $p_n\searrow1$ as $n\to\infty $. 
Then up to a subsequence and a translation if necessary,  we have $\lim\limits_{n\to\infty}e_2(p_n)=e_2^*$, and
\begin{equation}\label{2.5a}
u_{ip_n}\to u_i\ \  \text{strongly in}\ H^1(\R^d,\C)\ \ \text{as}\ \ n\to\infty,\ \   i=1,2,
\end{equation}
where $\gamma:=\sum_{i=1}^2|u_i\rangle\langle u_i|$ is a minimizer of the problem $e_2^*$ defined in \eqref{e}.
\end{lem}

\noindent \textbf{Proof.}
We shall carry out the proof by three steps.

$Step\ 1.$
We first claim that there exist a constant $R>0$  and a sequence $\{z_n\}\subset\R^d$ such that up to a subsequence  if necessary,
\begin{equation}\label{2.33a}
\lim\limits_{n\to\infty} \int_{B_R(z_n)}\rho_{\gamma_{p_n}}dx>0.
\end{equation}
To prove the above claim, recall from \cite{ho} the following Hoffmann-Ostenhof inequality
\begin{equation}\label{ho}
\mathrm{Tr}(-\Delta \gamma)\geq\inte|\nabla\sqrt{\rho_\gamma}|^2dx, \ \ \forall\ \gamma\in\mathcal{P}_2.
\end{equation}
Applying  the identity in \eqref{2.15}, we then obtain that the sequence $\big\{\sqrt{\rho_{\gamma_{p_n}}}\big\}$ is  bounded uniformly  in $H^1(\R^d)$. Thus, if \eqref{2.33a} is false,   then  the vanishing lemma (cf. \cite[Lemma 1.21]{minimax}) yields that
 \begin{align}\label{2.35a}
 \rho_{\gamma_{p_n}}=|u_{1p_n}|^2+|u_{2p_n}|^2\to 0\ \ \text{strongly in}\,\ L^r(\R^d)\ \text{as}\  n\to\infty,\ \   \forall\  r\in(1,\,  2^*/2),
 \end{align}
where we define $2^*=\frac{2d}{d-2}$ for $d\geq3$ and $2^*=\infty$ for $d=1,2$.
Applying \eqref{2.15}, there exists a constant $C>0$, independent of $n>0$, such that up to a subsequence  if necessary,
\begin{align}\label{2.39}
\arraycolsep=1.5pt
\left\{\begin{array}{lll}
\lim\limits_{n\to\infty}|\mu_{ip_n}|\alpha_{p_n}^{-d(p_n-1)}=1, \ \  i=1,2;\\[3mm]
\lim\limits_{n\to\infty}|\mu_{ip_n}|^{-\frac{1}{p_n-1}}\alpha_{p_n}^d
\leq \lim\limits_{n\to\infty}\big[1-C(p_n-1)\big]^{-\frac{1}{p_n-1}}<\infty,\ \   i=1,2.
\end{array}\right.
\end{align}
Applying \eqref{2.35a}, this then gives from \eqref{equ} and \eqref{2.28a} that for $i=1,2$,
\begin{align*}
0=&\inte|\nabla u_{ip_n}|^2dx-|\mu_{ip_n}|\alpha_{p_n}^{-d(p_n-1)}\inte\frac{|\mu_{ip_n}|^{-1}\alpha_{p_n}^{d(p_n-1)}\rho_{\gamma_{p_n}}^{p_n-1}-1}{p_n-1}u_{ip_n}^2dx\\
\geq&\inte|\nabla u_{ip_n}|^2dx-|\mu_{ip_n}|\alpha_{p_n}^{-d(p_n-1)}\\
&\  \cdot\int_{|\mu_{ip_n}|^{-1}\alpha_{p_n}^{d(p_n-1)}\rho_{\gamma_{p_n}}^{p_n-1}>1}
\frac{|\mu_{ip_n}|^{-1}\alpha_{p_n}^{d(p_n-1)}\rho_{\gamma_{p_n}}^{p_n-1}-1}{p_n-1}u_{ip_n}^2dx\\
\geq&\inte|\nabla u_{ip_n}|^2dx-C|\mu_{ip_n}|\alpha_{p_n}^{-d(p_n-1)}\big(|\mu_{ip_n}|^{-\frac{1}{p_n-1}}\alpha_{p_n}^{d}\big)^{\frac{2}{d}}\inte\rho_{\gamma_{p_n}}^{1+\frac{2}{d}}dx\geq o(1)\ \   \text{as}\ \, n\to\infty.
\end{align*}
We hence derive that $\lim\limits_{n\to\infty}\text{Tr}(-\Delta\gamma_{p_n})=0$, which however  contradicts with \eqref{2.15}. This therefore proves the claim \eqref{2.33a}.

Since the problem $e_2(p_n)$ is  translationally invariant, without loss of generality, we may assume $z_n=0$ in \eqref{2.33a}.
By the uniform  boundedness of $\{u_{ip_n}\}_n$  in $H^1(\R^d,\C)$ for $i=1,2$,  we then obtain that there exists a vector  $(u_1, u_2)\in \big(H^1(\R^d,\C)\big)^2\backslash\{(0, 0)\}$ such that up to a subsequence  if necessary,
\begin{eqnarray}\label{2.8}
u_{ip_n}\rightharpoonup  u_i\ \ \mathrm{weakly\ in}\, \ H^1(\R^d)\ \ \text{as}\ \ n\to\infty,\ \   i=1,2.
\end{eqnarray}
Set
\begin{eqnarray}\label{2.37}
\gamma:=\sum_{i=1}^{2}|u_i\rangle\langle u_i|\neq0.
\end{eqnarray}
If $\inte \rho_\gamma dx=2$,  then  we conclude from \eqref{2.8} that  
\begin{equation}\label{2.38}
\rho_{\gamma_{p_n}}\to\rho_\gamma\ \ \, \text{strongly in}\ \ L^r(\R^d)\ \ \text{as}\ \ n\to\infty, \ \   \forall\ r\in[1,\,  2^*/2),
\end{equation}
where we have used     Br\'ezis-Lieb Lemma  and the interpolation inequality.
Applying  Fatou's Lemma, this further implies from \eqref{2.28a} that
\begin{align}\label{2.40}
\lim\limits_{n\to\infty}\frac{1}{p_n(p_n-1)}\int_{\{\rho_{\gamma_{p_n}}(x)\geq1\}} \big(\rho_{\gamma_{p_n}}^{p_n}-\rho_{\gamma_{p_n}}\big)dx=\int_{\{\rho_\gamma(x)\geq1\}}\rho_\gamma\text{log}\rho_\gamma dx,
\end{align}
and
\begin{align}\label{2.41}
\liminf\limits_{n\to\infty}\frac{-1}{p_n(p_n-1)}\int_{\{\rho_{\gamma_{p_n}}(x)<1\}} \big(\rho_{\gamma_{p_n}}^{p_n}-\rho_{\gamma_{p_n}}\big)dx\geq-\int_{\{\rho_\gamma(x)<1\}}\rho_\gamma\text{log}\rho_\gamma dx.
\end{align}
Following Lemma \ref{lemstri}, we thus deduce  from \eqref{2.38}--\eqref{2.41} that  up to a subsequence if necessary,
\begin{align*}
\mathcal{E}(\gamma)\geq& e_2^*\geq\lim\limits_{n\to\infty}e_{2}(p_n)\\
=&\lim\limits_{n\to\infty}\Big[\text{Tr}(-\Delta\gamma_{p_n})-\frac{1}{p_n(p_n-1)}\inte \big(\rho_{\gamma_{p_n}}^{p_n}-\rho_{\gamma_{p_n}}\big)dx\Big]\\
\geq&\text{Tr}(-\Delta\gamma)-\inte \rho_\gamma\text{log}\rho_\gamma dx=\mathcal{E}(\gamma),
\end{align*}
which then yields that $\gamma$ is a minimizer of   $e_2^*$, and
\begin{equation*}
u_{ip_n}\to u_i\ \ \ \text{strongly in}\ \, H^1(\R^d,\C)\ \ \text{as}\ \ n\to\infty,\ \   i=1,2.
\end{equation*}

$Step\ 2$.  In order to complete the proof of Lemma \ref{prop},  we conclude from above that it next  suffices to prove $\inte \rho_\gamma dx=2$.  Actually, applying an adaptation of the classical dichotomy result (cf. \cite[Lemma 26]{ge}, \cite[Section 3.3]{begain}), we deduce from above that there exists a sequence $\{R_n\}\subset\R$, where $R_n\to\infty$ as $n\to\infty$, such that up to a subsequence  if necessary,
\begin{equation}\label{17}
0<\lim\limits_{n\to\infty}\int_{|x|\leq R_n}\rho_{\gamma_{p_n}}dx=\inte\rho_\gamma dx,\ \  \lim\limits_{n\to\infty}\int_{R_n\leq|x|\leq 6R_n}\rho_{\gamma_{p_n}}dx=0.
\end{equation}
Choose a cut-off function $\chi\in C_0^\infty(\R^d)$ such that $0\leq\chi\leq1$ in $\R^d$, where $\chi(x)=1$ holds for $|x|\leq1$, and $\chi(x)=0$ holds for $|x|\geq2$. Denote
\begin{equation}\label{2.41a}
\chi_{R_n}(x):=\chi(x/R_n),\ \  \eta_{R_n}(x):=\sqrt{1-\chi_{R_n}^2(x)}.
\end{equation}
It then follows from  the IMS formula (cf. \cite[Theorem 3.2]{ims}) that
\begin{align}\label{2.42}
\text{Tr}(-\Delta\gamma_{p_n})\geq \text{Tr}(-\Delta\chi_{R_n} \gamma_{p_n}\chi_{R_n})+\text{Tr}(-\Delta\eta_{R_n}\gamma_{p_n}\eta_{R_n})+o(1)\ \   \text{as}\ \ n\to\infty.
\end{align}
Moreover,   the Mean Value Theorem gives that  for any $a\in[0,1]$ and  $t\in[0,1]$, there exists  a constant $\theta\in(0,1)$ such that
\begin{align}\label{2.43a}
0\leq a^2(1-a^t)=-ta^{2+\theta t}\text{log}a\leq ta^{3/2}.
\end{align}
Employing  the interpolation inequality and the uniform boundedness of $\{\sqrt{\rho_{\gamma_{p_n}}}\}$ in $H^1(\R^d)$,
we then calculate  from \eqref{17} and \eqref{2.43a} that
\begin{align}\label{2.43}
&\inte\rho_{\gamma_{p_n}}^{p_n}dx\nonumber\\
=&\inte(\chi_{R_n}^2\rho_{\gamma_{p_n}})^{p_n}dx+\inte(\eta_{R_n}^2\rho_{\gamma_{p_n}})^{p_n}dx\nonumber\\
&+\inte\Big[\chi_{R_n}^2\rho_{\gamma_{p_n}}^{p_n}-\big(\chi_{R_n}^2\rho_{\gamma_{p_n}}\big)^{p_n}\Big]dx
+\inte\Big[\eta_{R_n}^2\rho_{\gamma_{p_n}}^{p_n}-\big(\eta_{R_n}^2\rho_{\gamma_{p_n}}\big)^{p_n}\Big]dx\nonumber\\
=&\inte(\chi_{R_n}^2\rho_{\gamma_{p_n}})^{p_n}dx+\inte(\eta_{R_n}^2\rho_{\gamma_{p_n}})^{p_n}dx+
\int_{R_n\leq|x|\leq2R_n}\chi_{R_n}^2\rho_{\gamma_{p_n}}^{p_n}(1-\chi_{R_n}^{2p_n-2})dx\nonumber\\
&
+\int_{R_n\leq|x|\leq2R_n}\eta_{R_n}^2\rho_{\gamma_{p_n}}^{p_n}(1-\eta_{R_n}^{2p_n-2})dx\\
\leq&\inte(\chi_{R_n}^2\rho_{\gamma_{p_n}})^{p_n}dx+\inte(\eta_{R_n}^2\rho_{\gamma_{p_n}})^{p_n}dx\nonumber\\
&
+2(p_n-1)\int_{R_n\leq|x|\leq2R_n}(\chi_{R_n}^{3/2}+\eta_{R_n}^{3/2})\rho_{\gamma_{p_n}}^{p_n}dx\nonumber\\
=&\inte(\chi_{R_n}^2\rho_{\gamma_{p_n}})^{p_n}dx+\inte(\eta_{R_n}^2\rho_{\gamma_{p_n}})^{p_n}dx
+o(p_n-1)\ \  \text{as}\ \ n\to\infty.\nonumber
\end{align}
As a consequence of \eqref{2.42} and \eqref{2.43},
we  derive that
\begin{align}\label{2.31a}
e_2(p_n)=\mathcal{E}_{p_n}(\gamma_{p_n})
\geq&\text{Tr}(-\Delta\chi_{R_n} \gamma_{p_n}\chi_{R_n})+F(\chi_{R_n}^2\rho_{\gamma_{p_n}})\\
&+\text{Tr}(-\Delta\eta_{R_n}\gamma_{p_n}\eta_{R_n})+F(\eta_{R_n}^2\rho_{\gamma_{p_n}})+o(1)\ \  \text{as}\ \ n\to\infty,\nonumber
\end{align}
where the function $F(\cdot)$ is as in \eqref{f}.

We now denote
\begin{align}\label{psi}
\Psi_n(x_1, x_2):=\frac{1}{\sqrt{2}}\big[u_{1p_n}(x_1)u_{2p_n}(x_2)-u_{1p_n}(x_2)u_{2p_n}(x_1)\big]\in H_a^1(\R^{2d},\C),
\end{align}
and let
$G_k^{\chi, n}=(G_k^{\chi, n})^*\geq0$ be a trace-class operator on $L^2(\R^{kd},\C)$, which satisfies
\begin{align}\label{g}
\arraycolsep=1.5pt
\left\{\begin{array}{lll}
G_0^{\chi, n}=\int_{\R^d\times\R^d}\big[1-\chi_{R_n}^2(x_1)\big]\big[1-\chi_{R_n}^2(x_2)\big]\, \big|\Psi_n(x_1,x_2)\big|^2dx_1dx_2,\\[3mm]
G_1^{\chi, n}(x_1;\, y_1)=2\chi_{R_n}(x_1)\chi_{R_n}(y_1)\int_{\R^d}\big[1-\chi_{R_n}^2(x_2)\big]\Psi_n(x_1, x_2)\overline{\Psi_n(y_1, x_2)}dx_2,\\[3mm]
G_2^{\chi, n}(x_1, x_2;\, y_1, y_2)=\chi_{R_n}(x_1)\chi_{R_n}(x_2)\chi_{R_n}(y_1)\chi_{R_n}(y_2)\Psi_n(x_1, x_2)\overline{\Psi_n(y_1, y_2)}.
\end{array}\right.
\end{align}
Applying \eqref{17}, we then derive from \eqref{psi} and \eqref{g} that
\begin{equation}\label{2.50a}
\sum_{k=0}^2\text{Tr}(G_k^{\chi, n})=1,\ \   \forall\ n>0,
\end{equation}
  and
\begin{equation}\label{2.50}
\begin{split}
\sum_{k=0}^2k\text{Tr}(G_k^{\chi, n})=&2\int_{\R^d\times \R^d} \chi_{R_n}^2(x_1)\big|\Psi_n(x_1, x_2)\big|^2dx_1dx_2\\
=&\inte \chi_{R_n}^2\rho_{\gamma_{p_n}}dx=\inte\rho_\gamma dx+o(1)
\ \ \, \text{as}\ \ n\to\infty.
\end{split}
\end{equation}

$Step\ 3.$ In order to prove that $\inte\rho_\gamma dx=2$, following \eqref{2.50a} and \eqref{2.50}, it still remains to prove the claim that    up to a subsequence  if necessary,
\begin{align}\label{2.50b}
\lim\limits_{n\to\infty}\text{Tr}(G_0^{\chi, n})=\lim\limits_{n\to\infty}\text{Tr}(G_1^{\chi, n})=0, \ \ \ \lim\limits_{n\to\infty}\text{Tr}(G_2^{\chi, n})=1.
\end{align}
To prove the claim (\ref{2.50b}), since $F(0)=0$ and the functional $F(\rho)$ is concave in $\rho\ge 0$, the same argument of  \cite[Equation (86)]{ge} first gives from \eqref{eq} and \eqref{2.31a} that
\begin{align}\label{2.35}
e_2(p_n)
\geq& \sum_{k=1}^2\text{Tr}(G_k^{\chi, n})\Big[ \text{Tr}\big(H(k)\tilde{G}_k^{\chi, n}\big)+F(\rho_{\tilde{G}_k^{\chi, n}})\Big]\nonumber\\
&+\sum_{k=1}^2\text{Tr}(G_k^{\eta, n})\Big[ \text{Tr}\big(H(k)\tilde{G}_k^{\eta, n}\big)+F(\rho_{\tilde{G}_k^{\eta, n}})\Big]+o(1)\\
\geq& \sum_{k=1}^2\text{Tr}(G_k^{\chi, n})e_k(p_n)+\sum_{k=1}^2\text{Tr}(G_k^{\eta, n})e_k(p_n)+o(1)
\ \ \ \text{as}\ \, n\to\infty,\nonumber
\end{align}
where  $\tilde{G}_k^{\chi, n}:=G_k^{\chi, n}/\text{Tr}(G_k^{\chi, n})$, and
$$
\rho_{\tilde{G}_{1}^{\chi, n}}(x):=\tilde{G}_{1}^{\chi, n}(x; x),\ \ \ \rho_{\tilde{G}_{2}^{\chi, n}}(x):=2\inte\tilde{G}_{2}^{\chi, n}(x, x_2;x, x_2)dx_2.
$$
Denote $e_0(p):=0$ and note from \eqref{g} that $\text{Tr}(G_k^{\eta, n})=\text{Tr}(G_{2-k}^{\chi, n})$ for $k=1,2$. We then conclude from \eqref{2.35} that
\begin{equation}\label{2.36}
\begin{split}
e_2(p_n)\geq&\sum_{k=1}^2\text{Tr}(G_k^{\chi, n})e_k(p_n)+\sum_{k=0}^1\text{Tr}(G_k^{\chi, n})e_{2-k}(p_n)+o(1)\\
=&\sum_{k=0}^2\text{Tr}(G_k^{\chi, n})\big[e_k(p_n)+e_{2-k}(p_n)\big]+o(1)\ \   \text{as}\ \, n\to\infty.
\end{split}\end{equation}
If  $\liminf\limits_{n\to\infty}\text{Tr}(G_1^{\chi, n})>0$, then it follows from \eqref{2.50a} and Lemma \ref{lemstri} that  up to a subsequence  if necessary,
\begin{equation*}
\begin{split}
&\lim\limits_{n\to\infty}\Big\{\sum_{k=0}^2\text{Tr}(G_k^{\chi, n})\big[e_k(p_n)+e_{2-k}(p_n)\big]-e_2(p_n)\Big\}\\
=&\lim\limits_{n\to\infty}\Big\{\sum_{k=0}^2\text{Tr}(G_k^{\chi, n})\big[e_k(p_n)+e_{2-k}(p_n)\big]-\sum_{k=0}^2\text{Tr}(G_k^{\chi, n})e_2(p_n)\Big\}\\
=&\lim\limits_{n\to\infty}\Big\{\text{Tr}(G_1^{\chi, n})\big[2e_1(p_n)-e_2(p_n)\big]\Big\}>0,
\end{split}\end{equation*}
which however contradicts  with \eqref{2.36}. We thus obtain that  $\liminf\limits_{n\to\infty}\text{Tr}(G_1^{\chi, n})=0$.

Up to a subsequence if necessary, we can now assume that  $\lim\limits_{n\to\infty}\text{Tr}(G_1^{\chi, n})=0$. Denote
\begin{align*}
\arraycolsep=1.5pt
\left\{\begin{array}{lll}
\Psi_{0n}(x_1, x_2):=\big(\text{Tr}(G_0^{\chi, n})\big)^{-\frac{1}{2}}\sqrt{1-\chi_{R_n}^2(x_1)}\sqrt{1-\chi_{R_n}^2(x_2)}\Psi_n(x_1, x_2),\\[2mm]
\Psi_{2n}(x_1, x_2):=\big(\text{Tr}(G_2^{\chi, n})\big)^{-\frac{1}{2}}\chi_{R_n}(x_1)\chi_{R_n}(x_2)\Psi_n(x_1, x_2),
\end{array}\right.
\end{align*}
and   the integral kernel of the one-particle density matrix $\gamma_\Psi$ associated with $\Psi\in H^1_a(\R^{2d},\C)$ is defined by
$$\gamma_\Psi(x, y):=2\inte \Psi(x, x_2)\overline{\Psi(y, x_2)}dx_2.
$$
It then follows from  \eqref{psi} and \eqref{g} that  $\gamma_{\Psi_n}$, $\gamma_{\Psi_{0n}}$ and  $\gamma_{\Psi_{2n}}$ are projectors, i.e.,
\begin{align}\label{gp}
\gamma_{\Psi_n}=\gamma_{\Psi_n}^2,\ \    \gamma_{\Psi_{0n}}=\gamma_{\Psi_{0n}}^2,\ \  \gamma_{\Psi_{2n}}=\gamma_{\Psi_{2n}}^2.
\end{align}
Since  
 $\lim\limits_{n\to\infty}\text{Tr}(G_1^{\chi, n})=0$,  one can calculate that
\begin{equation}\label{2.54}
\begin{split}
&\lim\limits_{n\to\infty}\big\|\gamma_{\Psi_n}-\text{Tr}(G_0^{\chi, n})\gamma_{\Psi_{0n}}-\text{Tr}(G_2^{\chi, n})\gamma_{\Psi_{2n}}\big\|_{\fS_{1}}=0,
\end{split}\end{equation}
and
\begin{align}\label{2.55}
0\leq&\lim\limits_{n\to\infty}\text{Tr}(G_0^{\chi, n})\text{Tr}(G_2^{\chi, n})\text{Tr}\big(\gamma_{\Psi_{0n}}\gamma_{\Psi_{2n}}\big)
\leq\lim\limits_{n\to\infty}\big(\text{Tr}(G_1^{\chi, n})\big)^2=0,
\end{align}
where $\fS_{1}$ denotes the space of trace-class operators on $L^2(\R^d,\C)$.  We thus deduce from \eqref{2.54}  that
\begin{align*}
\text{Tr}(G_0^{\chi, n})\gamma_{\Psi_{0n}}+\text{Tr}(G_2^{\chi, n})\gamma_{\Psi_{2n}}+o(1)=\gamma_{\Psi_n}\ \   \text{as}\ \ n\to\infty,
\end{align*}
which  further implies from   \eqref{gp} and \eqref{2.55} that
\begin{equation}\label{2.59}
\begin{split}
2=&\lim\limits_{n\to\infty}\text{Tr}(\gamma_{\Psi_n})
=\lim\limits_{n\to\infty}\text{Tr}(\gamma_{\Psi_n}^2)\\
=&\lim\limits_{n\to\infty}\text{Tr}\Big[\big(\text{Tr}(G_0^{\chi, n})\big)^2\gamma_{\Psi_{0n}}+\big(\text{Tr}(G_2^{\chi, n})\big)^2\gamma_{\Psi_{2n}}\Big]\\
=&\lim\limits_{n\to\infty}2\Big[\big(\text{Tr}(G_0^{\chi, n})\big)^2+\big(\text{Tr}(G_2^{\chi, n})\big)^2\Big].
\end{split}\end{equation}
Since $\lim\limits_{n\to\infty}\text{Tr}(G_1^{\chi, n})=0$,
it follows from \eqref{2.50a}, \eqref{2.50} and \eqref{2.59} that
$$
\lim\limits_{n\to\infty}2\text{Tr}(G_2^{\chi, n})=\inte\rho_\gamma dx>0,
$$
and
\begin{align*}
\lim_{n\to\infty}\big[\text{Tr}(G_0^{\chi,n})+
\text{Tr}(G_2^{\chi,n})\big]&=1=
\lim_{n\to\infty}\left[\big(\text{Tr}(G_0^{\chi,n})\big)^2+
\big(\text{Tr}(G_2^{\chi,n})\big)^2\right].
\end{align*}
Consequently, we have
$\lim\limits_{n\to\infty}\text{Tr}(G_0^{\chi,n})\text{Tr}(G_2^{\chi,n})=0$.
Since the limit $\lim\limits_{n\to\infty}2\text{Tr}(G_2^{\chi,n})$ is positive, it follows
that $\lim\limits_{n\to\infty}\text{Tr}(G_0^{\chi,n})=0$ and
$\lim\limits_{n\to\infty}\text{Tr}(G_2^{\chi,n})=1$, which hence proves the claim \eqref{2.50b}. The proof of Lemma \ref{prop} is therefore complete.  \qed

\subsection{Proof of Theorem  \ref{thm1}}\label{sec3}

In this subsection, we are ready to finish the proof of Theorem \ref{thm1}.

\vspace{0.15cm}

\noindent \textbf{Proof of Theorem \ref{thm1}.}
(1). Let $e_2(p)$ be as in \eqref{2.4}. We deduce from Lemmas \ref{lemstri} and   \ref{prop} that there exist two sequences $\{p_{n_k}\}$ and $\{p_{n_l}\}$ satisfying $\lim\limits_{n_k\to\infty}p_{n_k}=1^+$ and $\lim\limits_{n_l\to\infty}p_{n_l}=1^+$ such that
$$
0<e_2^*=\lim\limits_{n_k\to\infty}e_2(p_{n_k})=\liminf\limits_{p\searrow1}e_2(p)\leq\limsup\limits_{p\searrow1}e_2(p)=\lim\limits_{n_l\to\infty}e_2(p_{n_l})= e_2^*,
$$
which then yields that
\begin{equation}\label{3.6a}
	\lim\limits_{p\searrow1}e_2(p)=e_2^*>0.
\end{equation}
This proves Theorem \ref{thm1} (1).

(2). Let $\gamma_{p_n}=\sum_{i=1}^2|u_{ip_n}\rangle\langle u_{ip_n}|$ and $\gamma=\sum_{i=1}^2|u_i\rangle\langle u_i|$  be given by  \eqref{2.5a}.
We first claim that there exist constants $-\infty<\mu_1\leq\mu_2<\infty$ such that
\begin{align}\label{3.02}
-\Delta u_i=u_i\text{log}\rho_\gamma+\mu_iu_i\ \   \text{in}\ \, \R^d,\ \   i=1, 2.
\end{align}
Actually, recall  from \eqref{equ} that
\begin{align}\label{3.01a}
-\Delta u_{ip_n}=&\frac{\rho_{\gamma_{p_n}}^{p_n-1}-1}{p_n-1}u_{ip_n}
+\frac{1+\mu_{ip_n}\alpha_{p_n}^{-d(p_n-1)}}{p_n-1}u_{ip_n}
 \ \ \text{in}\ \, \R^d,\ \   i=1, 2.
\end{align}
It can be verified from \eqref{2.28a} and \eqref{2.5a} that
$$
\rho_{\gamma_{p_n}}\to\rho_\gamma\ \ \text{strongly in}\ L^r(\R^d)\ \text{as}\ n\to\infty,\ \ \forall\ r\in[1,\, 2^*/2),
$$
and for sufficiently large $n>0$,
\begin{align*}
0\leq\frac{\rho_{\gamma_{p_n}}^{p_n-1}-1}{p_n-1}|u_{ip_n}|\leq C(p_0)\rho_{\gamma_{p_n}}^{p_0}\ \ \text{in}\ \big\{x\in\R^d:\, \rho_{\gamma_{p_n}}(x)\geq1\big\},
\end{align*}
where the constant $p_0\in \big(1/2, \, 2^*/2\big)$, and $C(p_0)>0$ depends only on $p_0$. We then deduce that for any $\varphi\in C_c^\infty(\R^d,\C)$,
\begin{align}\label{3.4}
\lim\limits_{n\to\infty}\int_{\rho_{\gamma_{p_n}}\geq1}\frac{\rho_{\gamma_{p_n}}^{p_n-1}-1}{p_n-1}u_{ip_n}\bar{\varphi}dx
=\int_{\rho_{\gamma}\geq1}u_i\bar{\varphi}\, \text{log}\rho_\gamma dx, \ \ i=1,2.
\end{align}
Moreover, it follows from \cite[Lemma 2.1 (ii)]{wzq} that  if $\rho_{\gamma_{p_n}}(x)\leq1$, then we have
\begin{align*}
0\leq\frac{1-\rho_{\gamma_{p_n}}^{p_n-1}}{p_n-1}|u_{ip_n}|\leq C,\ \ i=1,2, 
\end{align*}
where  $C>0$ is independent of $x\in\R^d$ and $n>0$.
By the dominated convergence theorem, this yields  that for any $\varphi\in C_c^\infty(\R^d,\C)$,
\begin{align}\label{3.5}
\lim\limits_{n\to\infty}\int_{\rho_{\gamma_{p_n}}\leq1}\frac{\rho_{\gamma_{p_n}}^{p_n-1}-1}{p_n-1}u_{ip_n}\bar{\varphi}dx
=\int_{\rho_{\gamma}\leq1}u_i\bar{\varphi}\, \text{log}\rho_\gamma dx,\ \ i=1,2.
\end{align}
Applying \eqref{2.15}, we thus obtain  from  \eqref{3.4} and \eqref{3.5} that  \eqref{3.02} holds true.

We next prove that  up to a subsequence if necessary,
\begin{align}\label{3.2}
	u_{ip_n}\to u_i\ \   \text{strongly in}\ \, L^\infty(\R^d,\C)\ \, \text{as}\ \ n\to\infty,\ \   i=1,2.
\end{align}
Note from \eqref{3.01a} that
\begin{align}\label{3.01}
	-\Delta u_{ip_n}=&|\mu_{ip_n}|\alpha_{p_n}^{-d(p_n-1)}u_{ip_n}\frac{\big(|\mu_{ip_n}|^{\frac{-1}{p_n-1}}\alpha_{p_n}^{d}
\rho_{\gamma_{p_n}}\big)^{p_n-1}-1}{p_n-1}
	\ \,\ \text{in}\ \, \R^d,\ \,   i=1, 2.
\end{align}
Set  $\tilde{u}_{ip_n}:=a_0u_{ip_n}$ and  $\tilde{u}_i:=a_0u_i$ for some constant $a_0>0$. Following \eqref{2.39}, one can choose $a_0>0$ so that
\begin{align}\label{3.5c}
-2\text{log}a_0+\mu_i<0,\ \ \  a_0^{-2}\sup_{n}\Big(|\mu_{ip_n}|^{\frac{-1}{p_n-1}}\alpha_{p_n}^{d}\Big)<1,\ \ i=1, 2,
\end{align}
where $\mu_i$ is as in  \eqref{3.02}.
Applying \eqref{3.5c}  and Kato's inequality  (cf. \cite[Theorem X.27]{kato}), we hence deduce  from \eqref{3.02} and \eqref{3.01} that
\begin{align}\label{3.5a}
&-\Delta \big(|\tilde u_{1p_n}|+|\tilde u_{2p_n}|\big)
\leq C \big(|\tilde u_{1p_n}|+|\tilde u_{2p_n}|\big)\Big(\frac{\big(|\tilde u_{1p_n}|+|\tilde u_{2p_n}|\big)^{2p_n-2}-1}{p_n-1}\Big)_+\ \  \text{in}\ \, \R^d,
\end{align}
and
\begin{align}\label{3.6b}
-\Delta \big(|\tilde u_1|+|\tilde u_2|\big)\leq \big(|\tilde u_1|+|\tilde u_2|\big) \text{log} \big(|\tilde u_1|+|\tilde u_2|\big) ^2\ \  \text{in}\ \, \R^d,
\end{align}
where  $C:=\sup\limits_{n}|\mu_{1p_n}|\alpha_{p_n}^{-d(p_n-1)}\in(0, \infty)$.

Similar to  \cite[Lemma 3.10]{shuai}, we deduce from \eqref{3.02} and  \eqref{3.6b} that there exist constants $\alpha\in(0, 1)$, $C>0$ and $\theta>0$ such that
\begin{align}\label{3.8a}
u_i\in C^{2, \alpha}(\R^d,\C),\ \   |u_i(x)|\leq C e^{-\theta|x|}\ \  \text{in}\,\ \R^d,\ \   i=1,2.
\end{align}
Moreover, since the sequence $\{u_{ip_n}\}_n$  is bounded uniformly in $H^1(\R^d,\C)$ for $i=1,2$, the same argument of  \cite[Equation (2.11)]{zfl} gives from  \eqref{3.5a} that
\begin{align}\label{3.9a}
\sup\limits_{n}\big(\|u_{1p_n}\|_\infty+\|u_{2p_n}\|_\infty\big)<\infty.
\end{align}
Denote
$$f_{ip_n}:=\frac{\rho_{\gamma_{p_n}}^{p_n-1}-1}{p_n-1}u_{ip_n}
+\frac{1+\mu_{ip_n}\alpha_{p_n}^{-d(p_n-1)}}{p_n-1}u_{ip_n},\ \ i=1,2.$$
Note from \eqref{2.8a} and \eqref{2.28a} that there exists a constant $C>0$, independent of $n>0$,  such that for sufficiently large $n>0$,
\begin{align}\label{3.10b}
\frac{|\rho_{\gamma_{p_n}}^{p_n-1}-1|}{p_n-1}|u_{ip_n}|
\leq\rho_{\gamma_{p_n}}^{\frac{1}{2}}\big|\text{log}\rho_{\gamma_{p_n}}\big|+C \rho_{\gamma_{p_n}}
\leq 2C \big(\rho_{\gamma_{p_n}}^{\frac{1}{3}} +\rho_{\gamma_{p_n}} \big)\ \   \text{in}\ \, \R^d,
\end{align}
which then implies from \eqref{3.9a} that the sequence $\{f_{ip_n}\}_n$ is bounded uniformly in $L^q(\R^d,\C)$  for any $q\geq3$.

Using the $L^p$-theory (cf. \cite{elli})  and De Giorgi-Nash-Moser theory (cf. \cite[Theorem 4.1]{hq}), we deduce from above that  for any fixed $q>\max\{3, d/2\}$ and $ R>0$,
\begin{align}\label{3.10a}
\sup\limits_{n}\|u_{ip_n}\|_{W^{2,q}(B_R)}<\infty,
\end{align}
and
\begin{align}\label{3.11a}
\sup_{x\in B_{1}(y)} |u_{ip_n}(x)| \leq C\Big(\|f_{ip_n}\|_{L^q(B_2(y))} + \|u_{ip_n}\|_{L^2(B_2(y))}\Big),\ \  \forall\ y\in\R^d,
\end{align}
where the constant $C>0$ is independent of $n>0$ and $y\in\R^d$.
Since the embedding $W^{2,q}(B_R,\C)\hookrightarrow C(B_R,\C)$ is compact (cf.\cite[Theorem 7.26]{elli}) for any $q>d/2$, we conclude from  \eqref{2.5a} and \eqref{3.10a}  that there exists a subsequence, still denoted by $\{u_{ip_n}\}_{n}$, of $\{u_{ip_n}\}_{n}$ such that for any  fixed $R>0$,
\begin{equation}\label{conv1}
u_{ip_n}\to  u_i\ \ \mbox{strongly in}\ \ L^\infty(B_R,\C)\ \ \mathrm{as}\ n\to\infty,\ \   i=1,2.
\end{equation}
Applying \eqref{2.5a}, it follows from  \eqref{3.9a}, \eqref{3.10b} and \eqref{3.11a} that  for $i=1,2$,
\begin{equation}\label{conv2}
\lim_{|x| \to +\infty} |u_{ip_n}(x)| = 0\
\ \mathrm{uniformly\ for \ sufficiently\ large}\ n>0.
\end{equation}
As a consequence of  \eqref{conv1} and \eqref{conv2}, we  derive from \eqref{3.8a} that \eqref{3.2} holds true. This therefore completes the proof of Theorem  \ref{thm1}. \qed

\section{Qualitative Properties of Logarithmic  Fermionic Problems}\label{subsec3.2}
In this section, we shall study the qualitative properties of the logarithmic  fermionic problems $e_2^*$ and (\ref{logsobo}), based on which we finally address the proofs of Theorems  \ref{thm2} and \ref{thmlog}.

For convenience, we first define
\begin{align}\label{b1}
B(t):=t^2\text{log}t^2+A(t),\,\ t\geq0,
\end{align}
where
\begin{align}\label{a}
\arraycolsep=1.5pt
0\leq A(t):=
\left\{\begin{array}{lll}
-t^2\text{log}t^2,\ \,&\mathrm{if}\,\ 0\leq t\leq e^{-3},\\[3mm]
3t^2+4e^{-3}t-e^{-6}, \ \  &\mathrm{if}\,\  t\geq e^{-3}.
\end{array}\right.
\end{align}
One can check  that the function $A(t)$ is convex and increasing  in $t\geq0$, and there exists a constant $C(q)>0$ such that
\begin{align}\label{b}
0\leq B(t)\leq C(q)t^{2q}\ \,\ \text{in}\,\ \R^+,\ \  \forall\, q\in\big(1,\, 1+2/d\big).
\end{align}
Let $\big\{\gamma_n=\sum_{i=1}^2|u_{in}\rangle\langle u_{in}|\big\}$ be a minimizing sequence of the problem $e_2^*$ given by \eqref{e}. The following lemma shows that up to the translations, $\big\{(u_{1n}, u_{2n})\big\}$ is relatively compact in $\big(H^1(\R^d,\C)\big)^2$, which hence proves Theorem \ref{thm2} (1).

\begin{lem}\label{lem3.1}
Let $\big\{\gamma_n=\sum_{i=1}^2|u_{in}\rangle\langle u_{in}|\big\}\subset \mathcal{P}_2$ be a minimizing sequence of the problem $e_2^*$. Then up to a subsequence and  translations if necessary,
\begin{equation}\label{3.18}
u_{in}\to  u_i\ \ \mathrm{strongly\ in}\, \ H^1(\R^d, \C)\ \ \text{as}\ \ n\to\infty,\ \   i=1,2,
\end{equation}
where $\gamma=\sum_{i=1}^2|u_i\rangle \langle u_i|$ is a minimizer of $e_2^*$.
\end{lem}

\noindent \textbf{Proof.}
One can derive from \eqref{b1}--\eqref{b} that
\begin{align*}
e^*_2=&\text{Tr}(-\Delta\gamma_n) -\inte\rho_{\gamma_n}\text{log}\rho_{\gamma_n}dx+o(1)\\
=&\text{Tr}(-\Delta\gamma_n)-\inte B(\sqrt{\rho_{\gamma_n}})dx+\inte A(\sqrt{\rho_{\gamma_n}})dx+o(1)\\
\geq&\text{Tr}(-\Delta\gamma_n)-C(q)\inte \rho_{\gamma_n}^{q}dx+o(1)\\
\geq&\text{Tr}(-\Delta\gamma_n)-C(q)\big(K(d,q)\text{Tr}(-\Delta\gamma_n)\big)^{\frac{d(q-1)}{2}}+o(1)\ \ \ \text{as}\ \ n\to\infty,
\end{align*}
where we have used  (cf. \cite[Remark 7]{i}) the following Gagliardo-Nirenberg-Sobolev inequality for orthonormal systems
\begin{eqnarray}\label{gn}
\text{Tr}(-\Delta\gamma)\geq K(d,q)\Big(\int_{\R^d}\rho_\gamma^{q}dx\Big)^{\frac{2}{d(q-1)}},\ \   \forall\ \gamma\in\mathcal{P}_2.
\end{eqnarray}
This then yields that the sequence $\{\text{Tr}(-\Delta\gamma_n)\}$ is  bounded uniformly, and hence the sequence $\big\{\inte\rho_{\gamma_n}\text{log}\rho_{\gamma_n}dx\big\}$ is also bounded uniformly.

We first claim that
\begin{equation}\label{M3.18}
\text{the vanishing case of $\big\{\sqrt{\rho_{\gamma_n}}\big\}$ does not occur}.
\end{equation}
In fact,  since the sequences $\{\text{Tr}(-\Delta\gamma_n)\}$ and  $\big\{\inte\rho_{\gamma_n}\text{log}\rho_{\gamma_n}dx\big\}$ are bounded uniformly, we  deduce  from \eqref{2.8a} and \eqref{3.6a} that there exists a subsequence, still denoted by $\{\gamma_n\}$, of $\{\gamma_n\}$ such that
\begin{align*}
e^*_2=\lim\limits_{p\searrow1}e_2(p)
\leq&\lim\limits_{p\searrow1}\lim\limits_{n\to\infty}\Big[\text{Tr}(-\Delta\gamma_n)-\frac{1}{p(p-1)}\inte(\rho_{\gamma_n}^p-\rho_{\gamma_n})dx\Big]\nonumber\\
\leq&\lim\limits_{p\searrow1}\lim\limits_{n\to\infty}\Big[\text{Tr}(-\Delta\gamma_n)-\frac{1}{p}\inte\rho_{\gamma_n}\text{log}\rho_{\gamma_n}dx\Big]\\
=&\lim\limits_{n\to\infty}\Big[\text{Tr}(-\Delta\gamma_n)-\inte\rho_{\gamma_n}\text{log}\rho_{\gamma_n}dx\Big]=e^*_2,\nonumber
\end{align*}
which implies that
\begin{align}\label{3.1}
\lim\limits_{n\to\infty}\inte\rho_{\gamma_n}\text{log}\rho_{\gamma_n}dx=\lim\limits_{p\searrow1}\lim\limits_{n\to\infty}\frac{1}{p(p-1)}\inte(\rho_{\gamma_n}^p-\rho_{\gamma_n})dx.
\end{align}
If the vanishing occurs for the sequence $\{\sqrt{\rho_{\gamma_n}}\}$, then we have $\lim\limits_{n\to\infty}\inte\rho_{\gamma_n}^rdx=0$ for any $r\in(1, 2^*/2)$. This thus implies that
$$
\lim\limits_{p\searrow1}\lim\limits_{n\to\infty}\frac{1}{p(p-1)}\inte(\rho_{\gamma_n}^p-\rho_{\gamma_n})dx=-\infty,
$$
which however contradicts with \eqref{3.1}, in view of the fact that $\sup\limits_{n}\big|\inte\rho_{\gamma_n}\text{log}\rho_{\gamma_n}dx\big|<\infty$. This proves the claim (\ref{M3.18}).

Similar to \eqref{2.8} and \eqref{17},  since  $\big\{\sqrt{\rho_{\gamma_n}}=\sqrt{|u_{1n}|^2+|u_{2n}|^2}\big\}$ is non-vanishing and the problem $e_2^*$ is  translationally invariant, without loss of generality, we can assume that there exist $(u_1, u_2)\in \big(H^1(\R^d,\C)\big)^2\backslash\{(0, 0)\}$ and  a sequence $\{R_n\}\subset\R$ satisfying $R_n\rightarrow\infty$ as $n\to\infty$ such that up to a subsequence  if necessary,
\begin{eqnarray}\label{3.7}
u_{in}\rightharpoonup  u_i\ \ \mathrm{weakly\ in}\, \ H^1(\R^d,\C)\ \ \text{as}\ \ n\to\infty,\ \   i=1,2,
\end{eqnarray}
and
\begin{equation}\label{3.8}
0<\lim\limits_{n\to\infty}\int_{|x|\leq R_n}\rho_{\gamma_n}dx=\inte\rho_\gamma dx,\ \  \lim\limits_{n\to\infty}\int_{R_n\leq|x|\leq 6R_n}\rho_{\gamma_n}dx=0,
\end{equation}
where $\gamma:=\sum_{i=1}^2|u_i\rangle\langle u_i|$.
Let $\chi_{R_n}$ and $\eta_{R_n}$ be as in \eqref{2.41a}, where $R_n>0$ is as in \eqref{3.8}.
Since $\text{log}(1+t)\leq t$ holds for any $t\geq0$, one gets from \eqref{3.8} that
\begin{align*}
0\leq&\lim\limits_{n\to\infty}\inte\chi_{R_n}^2\rho_{\gamma_n}\Big[\text{log}\big(\chi_{R_n}^2\rho_{\gamma_n}+\eta_{R_n}^2\rho_{\gamma_n}\big)-\text{log}\big(\chi_{R_n}^2\rho_{\gamma_n}\big)\Big]dx\\
&+\lim\limits_{n\to\infty}\inte\eta_{R_n}^2\rho_{\gamma_n}\Big[\text{log}\big(\chi_{R_n}^2\rho_{\gamma_n}+\eta_{R_n}^2\rho_{\gamma_n}\big)-\text{log}\big(\eta_{R_n}^2\rho_{\gamma_n}\big)\Big]dx\\
=&\lim\limits_{n\to\infty}\int_{0\leq|x|\leq 2R_n}\chi_{R_n}^2\rho_{\gamma_n}\text{log}\big(1+\eta_{R_n}^2/\chi_{R_n}^2\big)dx\\
&+\lim\limits_{n\to\infty}\int_{R_n\leq|x|<\infty}\eta_{R_n}^2\rho_{\gamma_n}\text{log}\big(1+\chi_{R_n}^2/\eta_{R_n}^2\big)dx\\
\leq&2\lim\limits_{n\to\infty}\int_{R_n\leq|x|\leq 2R_n}\rho_{\gamma_n}dx=0,
\end{align*}
which then  gives that
\begin{align*}
&\inte\rho_{\gamma_n}\text{log}\rho_{\gamma_n}dx\\
=&\inte\big(\chi_{R_n}^2\rho_{\gamma_n}+\eta_{R_n}^2\rho_{\gamma_n}\big)\text{log}\big(\chi_{R_n}^2\rho_{\gamma_n}+\eta_{R_n}^2\rho_{\gamma_n}\big)dx\\
=&\inte\chi_{R_n}^2\rho_{\gamma_n}\text{log}\big(\chi_{R_n}^2\rho_{\gamma_n}\big)dx+
\inte\eta_{R_n}^2\rho_{\gamma_n}\text{log}\big(\eta_{R_n}^2\rho_{\gamma_n}\big)dx+o(1)\ \ \mbox{as}\,\ n\to\infty.
\end{align*}
By the IMS formula,  we then conclude from above that
\begin{align}\label{3.9}
e^*_2=&\text{Tr}(-\Delta\gamma_n) -\inte\rho_{\gamma_n}\text{log}\rho_{\gamma_n}dx+o(1)\nonumber\\
\geq&\Big[\text{Tr}(-\Delta\chi_{R_n}\gamma_n\chi_{R_n}) -\inte\chi_{R_n}^2\rho_{\gamma_n}\text{log}\big(\chi_{R_n}^2\rho_{\gamma_n}\big)dx\Big]\\
&+\Big[\text{Tr}(-\Delta\eta_{R_n}\gamma_n\eta_{R_n}) -\inte\eta_{R_n}^2\rho_{\gamma_n}\text{log}\big(\eta_{R_n}^2\rho_{\gamma_n}\big)dx\Big]+o(1)\ \ \mbox{as}\,\ n\to\infty.\nonumber
\end{align}
 Similar to  \eqref{2.50b},  one can deduce  from \eqref{3.9} that $\inte \rho_\gamma dx=2$,  and hence
\begin{equation}\label{3.10}
\rho_{\gamma_n}\to\rho_\gamma\ \   \text{strongly in}\ \, L^r(\R^d)\ \, \text{as}\ \, n\to\infty, \ \   \forall\ r\in[1, 2^*/2).
\end{equation}

Let $q\in(1, 1+2/d)$. Since there exists  a constant $C(q)>0$ such that $0\leq t\text{log}t\leq C(q)t^q$ holds for any $t>1$, one can  verify from \eqref{3.10} that
\begin{align}\label{3.11}
\infty>\lim\limits_{n\to\infty}\int_{\rho_{\gamma_n}(x)>1}\rho_{\gamma_n}\text{log}\rho_{\gamma_n}dx=\int_{\rho_\gamma(x)\geq1}\rho_{\gamma}\text{log}\rho_{\gamma}dx\geq0,
\end{align}
which then yields from Fatou's Lemma that
\begin{equation}\label{3.12}
\begin{split}
\infty>&\liminf\limits_{n\to\infty}\Big[e^*_2-\text{Tr}(-\Delta\gamma_n) +\int_{\rho_{\gamma_n}(x)>1}\rho_{\gamma_n}\text{log}\rho_{\gamma_n}dx\Big]\\
=&\liminf\limits_{n\to\infty}\int_{\rho_{\gamma_n}(x)\leq1}\rho_{\gamma_n}\big|\text{log}\rho_{\gamma_n}\big|dx\geq \int_{\rho_\gamma(x)\leq1}\rho_\gamma\big|\text{log}\rho_\gamma\big| dx\geq0.
\end{split}\end{equation}
It thus follows from \eqref{3.10}--\eqref{3.12} that  
\begin{align*}
\mathcal{E}(\gamma)\geq e_2^*=&\liminf\limits_{n\to\infty}\Big[\text{Tr}(-\Delta\gamma_n) -\int_{\rho_{\gamma_n}(x)\leq1}\rho_{\gamma_n}\text{log}\rho_{\gamma_n}dx-\int_{\rho_{\gamma_n}(x)>1}\rho_{\gamma_n}\text{log}\rho_{\gamma_n}dx\Big]\\
\geq&\text{Tr}(-\Delta\gamma)-\inte\rho_\gamma\text{log}\rho_\gamma dx=\mathcal{E}(\gamma).
\end{align*}
We therefore obtain that the convergence \eqref{3.18} holds true,  and we are done. \qed

\vspace{0.15cm}

Let $\gamma=\sum_{i=1}^2|u_i\rangle\langle u_i|$ be a minimizer of $e_2^*$.  To complete the proof of Theorem \ref{thm2}, we need to consider the corresponding Euler-Lagrange equation of $(u_1, u_2)$.
In fact, the definition of $e_2^*$ gives that
\begin{align*}
e^*_2=&\inf\Big\{\mathcal{E}(u_1, u_2):\,  u_i\in H^1(\R^d,\C),\ \langle u_i,u_j\rangle_{L^2}=\delta_{ij},\  i,j=1,2\Big\},
\end{align*}
where the energy functional $\mathcal{E}(u_1, u_2)$ satisfies
\begin{align}\label{fun}
\mathcal{E}(u_1, u_2):=\inte\big(|\nabla u_1|^2+|\nabla u_2|^2\big)dx-\inte \big(|u_1|^2+|u_2|^2\big)\text{log}\big(|u_1|^2+|u_2|^2\big)dx.
\end{align}
Since  $\inte \big(|\varphi|^2\text{log}|\varphi|^2\big)_+ dx<\infty$ holds for any $\varphi\in H^1(\R^d,\C)$, and it holds that 
\begin{align}\label{3.27}
\sum_{i=1}^2|u_i|^2 \text{log}|u_i|^2\leq&\big(|u_1|^2+|u_2|^2\big)\text{log}\big(|u_1|^2+|u_2|^2\big)\\
=&\sum_{i=1}^2|u_i|^2 \text{log}|u_i|^2
+|u_1|^2 \text{log}\Big(1+\frac{|u_2|^2}{|u_1|^2}\Big)
+|u_2|^2 \text{log}\Big(1+\frac{|u_1|^2}{|u_2|^2}\Big)\nonumber\\
\leq&\sum_{i=1}^2|u_i|^2 \text{log}|u_i|^2+\sum_{i=1}^2|u_i|^2\ \ \, \text{in}\ \, \R^d, \nonumber
\end{align}
one can  deduce  that  for any $\varphi\in H^1(\R^d,\C)$,
$$\inte \big(|u_1|^2+|u_2|^2\big)\Big|\text{log}\big(|u_1|^2+|u_2|^2\big)\Big|dx<\infty,$$
if and only if
$$\sum_{i=1}^2\inte |u_i|^2 \Big|\text{log}|u_i|^2\Big|dx<\infty.
$$

In order to ensure that  the energy functional $\mathcal{E}(u_1, u_2)$ is well-defined, one can observe from above that  the energy functional $\mathcal{E}(u_1, u_2)$ should be restricted to the set $\mathcal{W}\times \mathcal{W}$, where 
\begin{equation*}\label{w}
	\mathcal{W}:=\big\{u\in H^1(\R^d,\C): \inte |u|^2\big|\text{log}|u|^2\big|dx<\infty\big\},\ \  d\geq1.
\end{equation*}
Recall from \cite{regu1}  and  \cite{1983NA} that $\mathcal{W}$ is a reflexive Banach space  equipped  with the norm
\begin{align}\label{norm}
\|u\|_{\mathcal{W}}=\|u\|_{H^1}+\|u\|_{\mathcal{V}},
\end{align}
where the set
\begin{equation*}\label{v}
\mathcal{V}:=\Big\{u\in L^1_{\text{loc}}(\R^d,\C):\,  A(|u|)\in L^1(\R^d)\Big\}
\end{equation*}
is a reflexive Banach space equipped with the norm
\begin{align*}
\|u\|_{\mathcal{V}}=\inf\Big\{k>0:\ \inte A\big(k^{-1}|u|\big)dx\leq1\Big\},
\end{align*}
and the function $A$ is as in \eqref{a}.


\begin{lem}\label{lemc1}
Let  the functional $\mathcal{E}:\, \mathcal{W}\times \mathcal{W}\to \R$ be defined by  \eqref{fun}. Then we have $\mathcal{E}\in C^1(\mathcal{W}\times \mathcal{W}, \, \R)$. 
\end{lem}

\noindent \textbf{Proof.}  We first prove the claim that $\mathcal{E}\in C(\mathcal{W}\times \mathcal{W}, \R)$.
To prove this claim, we suppose  that
\begin{align*}
(u_{1n}, u_{2n}) \to (u_1, u_2)\ \ \text{strongly in}\,\ \mathcal{W}\times \mathcal{W}\,\ \text{as} \,\ n\to\infty.
\end{align*}
By  Sobolev's  embedding  theorem, we then obtain from \eqref{ho} and \eqref{norm} that the sequence $\big\{\sqrt{|u_{1n}|^2 + |u_{2n}|^2}\, \big\}$ is  bounded uniformly  in $H^1(\R^d)$,
\begin{align}\label{3.31}
\sqrt{|u_{1n}|^2 + |u_{2n}|^2}\to \sqrt{|u_1|^2 + |u_2|^2}\ \ \ \text{strongly in} \,\ L^r(\R^d)\ \ \text{as} \,\ n\to\infty, \ \ r\in[2, 2^*),
\end{align}
and
\begin{align}\label{3.32}
u_{in}\to u_i\ \ \ \text{strongly in} \ \mathcal{V}\,\ \text{as} \,\ n\to\infty,\ \ i=1,2.
\end{align}
Recall  that
\begin{equation}\label{3.33}
\begin{split}
&-(|u_{1n}|^2 + |u_{2n}|^2)\log(|u_{1n}|^2 + |u_{2n}|^2)\\
 =& A\big(\sqrt{|u_{1n}|^2 + |u_{2n}|^2}\big) - B\big(\sqrt{|u_{1n}|^2 + |u_{2n}|^2}\big),
\end{split}
\end{equation}
where the functions $A(\cdot)$ and $B(\cdot)$ are as in \eqref{b1} and \eqref{a}. Employing \cite[Equation (4.30)]{regu1} and \cite[Equation (1.2)]{1983NA}, it yields from \eqref{3.31} that
\begin{align}\label{3.35}
\lim\limits_{n\to\infty}\inte \Big|B\big(\sqrt{|u_{1n}|^2 + |u_{2n}|^2}\big)- B\big(\sqrt{|u_1|^2 + |u_2|^2}\big)\Big|dx=0.
\end{align}
Applying \cite[Proposition 6.1 (i)]{regu1} and \cite[Equations (1.5) and (2.3)]{1983NA}, since the function $A\geq 0$ is  convex and increasing on $[0, \infty)$, one can deduce  from \eqref{3.32} that
$$
0\leq A\big(\sqrt{|u_{1n}|^2 + |u_{2n}|^2}\big) \leq A\big(|u_{1n}| + |u_{2n}|\big) \leq 2A(|u_{1n}|) +2A(|u_{2n}|),
$$
and
$$
A(|u_{1n}|) +A(|u_{2n}|) \to A(|u_{1}|) +A(|u_{2}|)  \ \ \text{strongly in} \ L^1(\R^d) \,\ \text{as} \,\ n\to\infty.
$$
These then yield that
\begin{align}\label{3.36}
\lim\limits_{n\to\infty}\inte \Big|A\big(\sqrt{|u_{1n}|^2 + |u_{2n}|^2}\big)- A\big(\sqrt{|u_1|^2 + |u_2|^2}\big)\Big|dx=0.
\end{align}
We thus conclude from  \eqref{3.33}--\eqref{3.36} that the claim
$\mathcal{E}\in C(\mathcal{W}\times \mathcal{W},\R)$ holds true.

Moreover, we can regard \(\mathcal W\times\mathcal W\) as a real Banach space by
restricting its scalar field from \(\mathbb C\) to \(\mathbb R\), where the
  norm and completeness of \(\mathcal W\times\mathcal W\) are not changed.
Following the similar argument of \cite[Proposition 2.7]{1983NA} and
\cite[Proposition 6.3]{regu1}, it then yields that \(\mathcal E\) is G\^{a}teaux differentiable, in the sense that
\begin{align*}
	D\mathcal E(u_1,u_2)[h_1,h_2]
	={}&
	2\operatorname{Re}\sum_{i=1}^2
	\int_{\mathbb R^d}
	\nabla u_i\cdot\nabla\overline{h}_i\,dx
	\nonumber\\
	&-
	2\operatorname{Re}\sum_{i=1}^2
	\int_{\mathbb R^d}
	\bigl(1+\log(|u_1|^2+|u_2|^2)\bigr)
	u_i\overline{h}_i\,dx,
\end{align*}
and the map
\[
(u_1,u_2)\longmapsto D\mathcal E(u_1,u_2)\in(\mathcal W\times\mathcal W)^*
\]
is continuous from \(\mathcal W\times\mathcal W\) into its real
dual space, where the term $ u_i\log(|u_1|^2+|u_2|^2) $ ($i=1, 2$)
is defined to be zero on the set \(\{x\in \R^d:\, |u_1(x)|^2+|u_2(x)|^2=0\}\).  Hence the G\^{a}teaux derivative is the Fr\'echet derivative, which then yields that $\mathcal E\in
C^1(\mathcal W\times\mathcal W,\mathbb R).$
This completes the proof of Lemma~\ref{lemc1}. \qed

\vspace{0.1 cm}
Following Lemma~\ref{lemc1}, we next address the proof of Theorem \ref{thm2}.
\vspace{0.1 cm}

\noindent \textbf{Proof of Theorem \ref{thm2}.}  Since (1) of Theorem \ref{thm2} follows directly from Lemma \ref{lem3.1}, the rest is to prove (2) of Theorem \ref{thm2}.

Let $\gamma=\sum_{i=1}^2|u_i\rangle \langle u_i|$ be a minimizer of $e_2^*$. Since $\mathcal E\in C^1(\mathcal W\times\mathcal W,\, \R)$, the variational theory yields that up to a unitary transformation of  $(u_1,u_2)$, there exists $(\mu_1, \mu_2)$ satisfying $\mu_1\leq\mu_2$ such that
 \begin{align}\label{3.35b}
 \big[-\Delta -\log(|u_1|^2+|u_2|^2)\big]u_i=\mu_iu_i
 \ \ \text{in }\, \R^d,\ \,  i=1,2.
 \end{align}
Similar to \eqref{3.6b} and \eqref{3.8a},
one can further verify  from \eqref{3.35b} that
\begin{align}\label{3.35a}
u_i\in C^2(\R^d,\C),\ \   |u_i(x)|\leq C e^{-\theta|x|}\ \,\ \text{in}\,\ \R^d,\ \   i=1,2,
\end{align}
and
\begin{align}\label{3.36a}
-\Delta \big(|\tilde u_1|+|\tilde u_2|\big)\leq \big(|\tilde u_1|+|\tilde u_2|\big) \text{log} \big(|\tilde u_1|+|\tilde u_2|\big) ^2\ \   \text{in}\ \, \R^d,
\end{align}
where $\theta>0,\ C>0$ are independent of $x\in\R^d$, and $\tilde u_i:=a_0 u_i$ holds for some $a_0>e^{\mu_2/2}$. Applying the  comparison argument (cf. \cite[Theorem 1.4 (2)]{zfl}), we then  deduce  from \eqref{3.35a} and \eqref{3.36a} that there exist constants $\theta_0>0$ and $C_0>0$ such that
$$
|\tilde u_1(x)|+|\tilde u_2(x)|\leq C_0 e^{-\theta_0|x|^2}\ \ \text{in}\ \ \R^d.
$$

For simplicity, we define
\[f_+:=\max\{f,\, 0\},\ \ f_-:=\max\{-f,\, 0\}.\]
Following the following three steps, it next suffices to prove that $(u_1, u_2)$ is a ground state of the system \eqref{3.35b} in the sense of Definition \ref{dfn:1}.
To complete the proof of Theorem \ref{thm2}, it remains to show that $(u_1,u_2)$ is a ground state of the system \eqref{3.35b} in the sense of Definition \ref{dfn:1}. We next prove this result  in the following three steps.

$Step\ 1.$
Denote
$$
D(\mathfrak q_\gamma)
:=\big\{\varphi\in H_0^1(\Omega_\gamma,\mathbb C):\, (V_\gamma)_+^{1/2}\varphi\in L^2(\Omega_\gamma,\mathbb C)\big\},
$$
and define the symmetric sesquilinear form
$\mathfrak q_\gamma:D(\mathfrak q_\gamma)\times D(\mathfrak q_\gamma)\to\mathbb{C}$ by
\begin{equation*}\label{xj1}
	\mathfrak q_\gamma[\varphi,\, \psi]
	:
	=\int_{\Omega_\gamma}\nabla\overline{\varphi}\cdot\nabla\psi\,dx+\int_{\Omega_\gamma}V_\gamma\overline{\varphi}\psi\,dx,
\end{equation*}
where \[
  V_\gamma:=-\log\rho_\gamma=-\log(|u_1|^2+|u_2|^2)\ \,\ \hbox{in}\,\ \Omega_\gamma:=\big\{x\in\R^d:\ \rho_\gamma(x)>0\big\}.
\]
Note from \eqref{3.35a} that
\begin{equation}\label{v8}
\Omega_\gamma\ \ \text{is open}, \ \   V_\gamma\in L^\infty_{\mathrm{loc}}(\Omega_\gamma), \ \
	(V_\gamma)_-\in L^\infty(\Omega_\gamma).
\end{equation}
Let
\begin{align}\label{3.70}
	\|\varphi\|_{\mathfrak q}^2 := \mathfrak q_\gamma[\varphi, \, \varphi] + (M_\gamma+1)\|\varphi\|_{L^2(\Omega_\gamma)}^2,\   \ \text{where}\ \  M_\gamma:= \|(V_\gamma)_-\|_{L^\infty(\Omega_\gamma)}<\infty.
\end{align}
We then have
\begin{align}\label{3.70a}
\text{the norm}\  \|\varphi\|_{\mathfrak q}^2\ \text{is equivalent  to}\ \|\varphi\|_{H^1(\Omega_\gamma)}^2+\|(V_\gamma)_+^{1/2}\varphi\|^2_{L^2(\Omega_\gamma)}\ \text{on}\ D(\mathfrak q),
\end{align}
which implies that the space
$D(\mathfrak q_\gamma)$ is complete with respect to
$\|\cdot\|_{\mathfrak q}$.  Thus, one can verify from \eqref{v8} that $\mathfrak q_\gamma$ is densely defined, closed, symmetric and bounded from below. By the first representation theorem, then there exists a unique self-adjoint operator  
\begin{equation}\label{op0}
H_\gamma^D:\
D(H_\gamma^D)\subset L^2(\Omega_\gamma,\mathbb C)
 \rightarrow L^2(\Omega_\gamma,\mathbb C),
\end{equation}
which has a lower bound $-\|(V_\gamma)_-\|_{L^\infty(\Omega_\gamma)}>-\infty$,  such that  for any  $\psi\in D(H_\gamma^D)$ and $\varphi\in D(\mathfrak q_\gamma)$,
\begin{equation}\label{op}
	\mathfrak q_\gamma[\varphi,\psi]
	=\big\langle \varphi,\, H_\gamma^D\psi\big\rangle_{L^2(\Omega_\gamma)}.
\end{equation}
Here  $D(H_\gamma^D)$ denotes the domain of $H_\gamma^D$ satisfying $H_\gamma^D\psi:=f$ and
\begin{align}\label{d}
	D(H_\gamma^D)
	=\Big\{
	\psi\in D(\mathfrak q_\gamma):\,\ &
	\text{there exists }f\in L^2(\Omega_\gamma,\mathbb C)
	\text{ such that}
	\\[-2mm]
	&
	\mathfrak q_\gamma[\varphi,\, \psi]
	=
	\langle \varphi,\,  f\rangle_{L^2(\Omega_\gamma)}\ \text{for all}\ \varphi\in D(\mathfrak q_\gamma)
	\Big\}.
	\nonumber
\end{align}

Note from \eqref{op} that for any  $\psi\in D(H_\gamma^D)$ and
$\varphi\in C_c^\infty(\Omega_\gamma)$,
\begin{align}\label{3.72}
	\langle\varphi,\, H_\gamma^D\psi\rangle_{L^2(\Omega_\gamma)}
	&=
	\int_{\Omega_\gamma}
	\nabla\overline{\varphi}\cdot\nabla\psi\,dx
	+
	\int_{\Omega_\gamma}
	V_\gamma\overline{\varphi}\psi\,dx.
\end{align}
Since $V_\gamma\in L^\infty_{\mathrm{loc}}(\Omega_\gamma)$,
$V_\gamma\psi\in L^1_{\mathrm{loc}}(\Omega_\gamma)$ holds for any
$\psi\in D(H_\gamma^D)\subset H_0^1(\Omega_\gamma)$. This yields from \eqref{3.72} that
\begin{align}\label{di}
	H_\gamma^D\psi
	=
	(-\Delta+V_\gamma)\psi
	\quad\text{in }\ \mathcal D'(\Omega_\gamma),
\end{align}
where $\mathcal D'(\Omega_\gamma)$ denotes the space of the
distributions on $\Omega_\gamma$.
Since the Dirichlet boundary condition is encoded by
$D(\mathfrak q_\gamma)\subset H_0^1(\Omega_\gamma)$,
we conclude from \eqref{di} that $H_\gamma^D$ is  the Dirichlet realization in
$L^2(\Omega_\gamma,\mathbb C)$ of the differential expression
\[
-\Delta+V_\gamma
=
-\Delta-\log\rho_\gamma
\ \ \text{in }\, \Omega_\gamma.
\]

$Step\ 2.$  In this step, we prove that $\mu_1$ and $\mu_2$ are the first two (counted with multiplicity) eigenvalues of the operator $H_\gamma^D$, where  $-\infty<\mu_1\leq\mu_2<\infty$ and $H_\gamma^D$ are given by \eqref{3.35b} and \eqref{op0}, respectively.

We first prove that $H_\gamma^D$ has a compact resolvent.
Indeed, it follows from \eqref{3.35a} that there exist constants \(\delta>0\) and \(C_0>M_\gamma+1>0\) such that
\[
V_\gamma(x)+C_0\geq \delta(1+|x|)\ \   \text{in }\, \Omega_\gamma.
\]
This  yields that for any  $R>0$ and $\varphi\in D(\mathfrak q_\gamma)$,
\begin{align}\label{eq:tail-control}
	\int_{\Omega_\gamma\cap\{x: |x|>R\}}|\varphi|^2\,dx
	&\leq\frac{\delta^{-1}}{1+R}\int_{\Omega_\gamma}(V_\gamma+C_0)|\varphi|^2\,dx
	\leq \frac{C_\gamma}{1+R}\|\varphi\|^2_{\mathfrak q},
\end{align}
where $C_\gamma>0$ depends only on $C_0$, $\delta$ and $M_\gamma$.
Using the zero extension map
$H_0^1(\Omega_\gamma,\C)\hookrightarrow H^1(\mathbb R^d,\C)$ and
the compact embedding $H^1(B_R)\hookrightarrow L^2(B_R)$,     the uniform  estimate \eqref{eq:tail-control} then yields that \begin{align}\label{3.23a}
	D(\mathfrak q_\gamma),\ \text{equipped with  the norm}\ \|\cdot\|_{\mathfrak q},\ \text{ is compactly embedded into}\  L^2(\Omega_\gamma,\C),
\end{align}
where $\|\cdot\|_{\mathfrak q}$ is as in \eqref{3.70}. Since    $H_\gamma^D$   is a self-adjoint operator associated with  $\mathfrak q_\gamma$, we conclude from (\ref{3.23a}) that
$H_\gamma^D$ has a compact resolvent.

We now prove  that 
\begin{equation}\label{3.25}
	H_\gamma^D u_i=\mu_i u_i\ \ \text{in}\ \, L^2(\Omega_\gamma,\C),\ \ u_i\in D(H_\gamma^D),
	\ \ i=1,2.
\end{equation}
Since $u_i\in H^1(\mathbb R^d,\mathbb C)\cap C(\mathbb R^d,\mathbb C)$
and $u_i=0$ on $\Omega_\gamma^c$,   
the quasi-continuous characterization of
$H_0^1(\Omega_\gamma,\C)$  gives that
\[
u_i|_{\Omega_\gamma}\in H_0^1(\Omega_\gamma,\mathbb C),\ \   i=1,2.
\]
In what follows,   we always identify the functions of $H_0^1(\Omega_\gamma)$ with   zero extensions, so that they are defined on $\R^d$.
Since $\gamma$ is a minimizer of $e_2^*$,  one gets from \eqref{3.35a} that
\begin{align*}
	\int_{\Omega_\gamma}(V_\gamma)_+|u_i|^2dx
	&\leq \inte\rho_\gamma (\log \rho_\gamma)_-dx
	=\inte\rho_\gamma (\log \rho_\gamma)_+dx-\inte\rho_\gamma \log \rho_\gamma dx\\
	&\leq  2C_1-\inte\rho_\gamma \log \rho_\gamma dx =2C_1-\text{Tr}(-\Delta\gamma)+e_2^*<\infty.
\end{align*}
This further shows that $\int_{\Omega_\gamma}(V_\gamma)_+|u_i|^2\,dx<\infty$,  and hence  $u_i\in D(\mathfrak q_\gamma) $ for $i=1,2$.
Note from  \eqref{3.35b}  that 
\begin{equation}\label{3.26a}
	\begin{split}
		\mathfrak q_\gamma[\varphi,u_i]
		&=\int_{\Omega_\gamma}\nabla u_i\cdot\nabla\overline{\varphi}\,dx
		+\int_{\Omega_\gamma}V_\gamma u_i\overline{\varphi}\,dx\\
		&=\mu_i\int_{\Omega_\gamma}u_i\overline{\varphi}\,dx=\langle \varphi,\ \mu_iu_i\rangle_{L^2(\Omega_\gamma)},\ \   \forall\ \varphi\in C_c^\infty(\Omega_\gamma,\mathbb C).
	\end{split}
\end{equation}
Since $
(V_\gamma)_+
\in L^\infty_{\mathrm{loc}}(\Omega_\gamma)
\subset L^1_{\mathrm{loc}}(\Omega_\gamma),$
we obtain that $C_c^\infty(\Omega_\gamma,\mathbb C)$
is dense in $ D(\mathfrak q_\gamma)$, which is
equipped with the norm $
\Big(\|\varphi\|_{H^1(\Omega_\gamma)}^2
+\|(V_\gamma)_+^{1/2}\varphi\|_{L^2(\Omega_\gamma)}^2
\Big)^{1/2}$.
This yields from \eqref{3.70a} that  $C_c^\infty(\Omega_\gamma,\mathbb C)$ is dense in $D(\mathfrak q_\gamma)$ with respect to $\|\cdot\|_\mathfrak q$. We therefore deduce  from  \eqref{d} and \eqref{3.26a}  that  \eqref{3.25} holds true.



Since $H_\gamma^D$ is bounded from below and has compact resolvent, we obtain from  \cite[Theorem XIII.64]{modern4}  that the eigenvalues, counted
with multiplicity, of $H_\gamma^D$ can be rearranged as
\[
 \lambda_1(H_\gamma^D)\leq\lambda_2(H_\gamma^D)\leq\cdots\leq \lambda_k(H_\gamma^D)\leq\cdots,\ \   \lim\limits_{k\to\infty}\lambda_k(H_\gamma^D)=\infty.
\]
The min--max principle then gives that if $\mu_1$ and $\mu_2$ are
not the first two eigenvalues of $H_\gamma^D$, then there exist
$\mu_*\in\{\mu_1,\mu_2\}$, together with its corresponding eigenfunction
$u_*\in\{u_1,u_2\}$, and  a normalized eigenfunction $u\in D(H_\gamma^D)$ such that
\[
H_\gamma^D u=\mu u\ \ \ \text{in}\,\ L^2(\Omega_\gamma),\ \   \mu<\mu_*, \ \
u\perp\operatorname{span}\{u_1,u_2\}.
\]
Set $\gamma':=\gamma-|u_*\rangle\langle u_*|+|u\rangle\langle u|\in\mathcal{P}_2.$  Since there exists a constant $C_d>0$, depending only on  $d\in \mathbb{N}^+$, such that
\begin{align*}
	(s\log s)_+
	\leq C_d s^{1+\frac2d}\ \  \text{for any}\,\ s\geq0,
\end{align*}
we deduce from \eqref{lt} that
\begin{align}\label{3.83}
\int_{\Omega_\gamma}\big(\rho_{\gamma'}\log\rho_{\gamma'}\big)_+dx\leq& C_d\inte \rho_{\gamma'}^{1+\frac2d}dx\leq C_dc^{-1}_{LT}(d)\text{Tr}(-\Delta \gamma')<\infty,
\end{align}
where  $c_{LT}(d)>0$ is given by  \eqref{lt} and depends only on  $d\in \mathbb{N}^+$.
Since
\[
s\log\frac{s}{t}-s+t\geq0,\ \  \forall\ s\geq0,\  t>0,
\]
we have
\begin{align*}
\int_{\Omega_\gamma}\rho_{\gamma'}\log\frac{\rho_{\gamma'}}{\rho_\gamma}\,dx
=\int_{\Omega_\gamma}\Big[\rho_{\gamma'}\log\frac{\rho_{\gamma'}}{\rho_\gamma}-\rho_{\gamma'}+\rho_{\gamma}\Big]dx\geq0.
\end{align*}
which further gives that
\begin{align}\label{3.82}
\int_{\Omega_\gamma}\rho_{\gamma'}\log\rho_{\gamma'}dx
\geq\int_{\Omega_\gamma}\rho_{\gamma'}\log\rho_{\gamma}dx=-\int_{\Omega_\gamma}\rho_{\gamma'}V_\gamma dx
>-\infty,
\end{align}
due to the facts that $(V_\gamma )_-\in L^\infty(\Omega_\gamma)$ and $u_i, u\in D(\mathfrak q_\gamma)$.
We thus conclude  from \eqref{3.83} and \eqref{3.82} that   $-\infty<\int_{\Omega_\gamma}\rho_{\gamma'}\log\rho_{\gamma'}dx<\infty$.
By the convexity of  the map $t\mapsto t\log t$, we hence obtain that
\begin{align}\label{3.76}
e_2^*\leq&\mathcal{E}(\gamma')
=\mathcal{E}(\gamma)+\text{Tr}(-\Delta \gamma')-\text{Tr}(-\Delta \gamma)+\int_{\Omega_\gamma}\rho_\gamma\log\rho_\gamma dx-\int_{\Omega_\gamma}\rho_{\gamma'}\log\rho_{\gamma'} dx\nonumber\\
\leq&\mathcal{E}(\gamma)+\text{Tr}(-\Delta \gamma')-\text{Tr}(-\Delta \gamma)-\int_{\Omega_\gamma}\big(1+\log\rho_{\gamma} \big)\big(\rho_{\gamma'} -\rho_{\gamma} \big)dx\\
=&e_2^*-\int_{\Omega_\gamma}|\nabla u_*|^2dx+\int_{\Omega_\gamma}|\nabla u|^2dx
+\int_{\Omega_\gamma} \big(|u_*|^2-|u|^2\big)\log \rho_\gamma dx\nonumber\\
	=&e_2^*+(\mu-\mu_*)<e_2^*,\nonumber
\end{align}
a contradiction. This proves  that  $\mu_1$ and $\mu_2$ are the first two (counted with multiplicity) eigenvalues of $H_\gamma^D$, and we are done.

$Step\ 3.$
It remains to prove that the first eigenvalue $\mu_1$ of
$H_\gamma^D$ is simple. On the contrary, suppose
$\dim\ker(H_\gamma^D-\mu_1)\geq2$.
Since $V_\gamma$ is real-valued, the operator $H_\gamma^D$
commutes with the complex conjugation. Thus, we may  deal with the corresponding real eigenspace.
Choose a nonzero real-valued function
$f\in\ker(H_\gamma^D-\mu_1)$. Note from \cite[Section 6.17]{analysis} that
\[\mathfrak q_\gamma\big[|f|,\, |f|\big]=\mathfrak q_\gamma[f,\, f].\]
This yields that  the function $|f|$ also minimizes the Rayleigh quotient and
thus is an eigenfunction corresponding to the first eigenvalue $\mu_1$ of $H_\gamma^D$. Since $\dim\ker(H_\gamma^D-\mu_1)\geq2$, we can choose a  nonzero real-valued function $g\in\ker(H_\gamma^D-\mu_1)$ such that
\[
\big\langle g,\, |f|\big\rangle_{L^2(\Omega_\gamma)}=0.
\]
If $g$ has a fixed sign, then
\[
|f||g|=0
\quad\text{a.e. in }\,\Omega_\gamma.
\]
If $g$ changes the sign, then both $g_+$ and $g_-$ are nonzero. Note that
\begin{equation*}
	\mathfrak q_\gamma[g,g]
	=
	\mathfrak q_\gamma[g_+,g_+]
	+
	\mathfrak q_\gamma[g_-,g_-],
	\ \
	\|g\|_2^2
	=
	\|g_+\|_2^2+\|g_-\|_2^2.
\end{equation*}
The variational characterization then gives that
\[
\mathfrak q_\gamma[g_\pm,g_\pm]
=
\mu_1\|g_\pm\|_2^2,
\]
which further implies that both  $g_+$ and $g_-$ are the eigenfunctions corresponding to the first eigenvalue $\mu_1$ of $H_\gamma^D$.
Consequently,  there exist two normalized nonnegative
eigenfunctions $\phi_1$ and $\phi_2$, corresponding to the first eigenvalue $\mu_1$ of $H_\gamma^D$, such that
\begin{equation}\label{3.disjoint-groundstates}
	\phi_1\phi_2=0
	\quad\text{a.e. in }\, \Omega_\gamma.
\end{equation}

Denote \[
\widetilde\gamma
:=
|\phi_1\rangle\langle\phi_1|
+
|\phi_2\rangle\langle\phi_2|
\in\mathcal P_2.
\]
Since $\mu_1\leq\mu_2$ are the first two (counted with multiplicity) eigenvalues of
$H_\gamma^D$, we deduce that $\mu_1=\mu_2$, due to the fact that
$\dim\ker(H_\gamma^D-\mu_1)\geq2$ . 
The same argument of \eqref{3.76} then gives that
\begin{align}\label{3.78}
e_2^*\leq\mathcal E(\widetilde\gamma)
&\leq\mathcal E(\gamma)+\text{Tr}(-\Delta\widetilde\gamma)\nonumber\\
&\quad -\int_{\Omega_\gamma}\rho_{\widetilde\gamma}\text{log}\rho_{\gamma}dx-\text{Tr}(-\Delta\gamma)+
\int_{\Omega_\gamma}\rho_{\gamma}\text{log}\rho_{\gamma}dx\\
&=e_2^*+2\mu_1-2\mu_1=e_2^*.\nonumber
\end{align}
Moreover, it follows from \eqref{3.disjoint-groundstates} that
\[
\rho_{\widetilde\gamma}\log\rho_{\widetilde\gamma}
=
|\phi_1|^2\log|\phi_1|^2
+
|\phi_2|^2\log|\phi_2|^2
\quad\text{a.e. in }\, \mathbb R^d.
\]
We thus conclude from  \eqref{3.78} that
\begin{align*}
e_2^*=\mathcal E(\widetilde\gamma)
&=\sum_{j=1}^2\left[\int_{\mathbb R^d}|\nabla\phi_j|^2\,dx
-\int_{\mathbb R^d}|\phi_j|^2\log|\phi_j|^2\,dx\right]
\geq2e_1^*,
\end{align*}
where $e_1^*>0$ is given by \eqref{2.4b}.
This however contradicts  with Lemma \ref{lemstri}. Hence,
 the first eigenvalue $\mu_1$ of
$H_\gamma^D$ is simple, which proves Step 3. This therefore completes the proof  of Theorem \ref{thm2}.
\qed

\subsection{Proof of Theorem \ref{thmlog}}\label{subsec3.1}
This subsection is devoted to the proof of Theorem  \ref{thmlog},  which concerns the sharp logarithmic Sobolev inequality  \eqref{logsobo} for orthonormal functions. 

We define
\begin{align}\label{k2}
	\mathcal{K}_2
	:=
	\big\{
	\gamma:\,
	0\leq \gamma=\gamma^*\leq 1,\
	\Tr(\gamma)=2,\
	\Tr(-\Delta\gamma)<\infty
	\big\}.
\end{align}
By the spectral theorem of \cite{i} and the references therein,  for any $\gamma\in\mathcal{K}_2$, there exist an orthonormal basis $\{u_i\}_{i\geq1}$ of $L^2(\R^d, \C)$ and a sequence $\{n_i\}_{i\geq1}\subset[0,1]$ satisfying $\sum_{i\geq1}n_i=2$ such that
\begin{equation}\label{3.16}
	\gamma
	=
	\sum_{i\geq1}
	n_i|u_i\rangle\langle u_i|,
\end{equation}
in the sense that
\begin{equation*}
	(\gamma\varphi)(x)=\sum_{i\geq1}n_iu_i(x)\langle u_i,\, \varphi\rangle_{L^2},\ \ \forall\ \varphi\in L^2(\R^d, \C).
\end{equation*}
We also denote for  $\gamma\in\mathcal{K}_2$,
\begin{align}\label{1.4b}
	\text{Tr}(-\Delta\gamma):= \sum_{i\geq1}\inte n_i|\nabla u_i|^2dx,\  \   \rho_\gamma:=\sum_{i\geq1}n_i|u_i|^2.
\end{align}
We first establish the following decomposition result for the operators in $\mathcal{K}_2$.

\begin{lem}\label{lem-projector-decomposition}
Suppose $\gamma\in\mathcal{K}_2$, where $\mathcal{K}_2$ is given by \eqref{k2} and $\mathcal{P}_2$ is as in \eqref{pl} with $N=2$. Then there exists a measurable family
	$\{\gamma_t\}_{t\in[0,1)}\subset\mathcal{P}_2$ such that
\begin{align}\label{kinetic-decomposition}
	\Tr(-\Delta\gamma)
	=
	\int_0^1\Tr(-\Delta \gamma_t)\,dt,
\end{align}
and
\begin{align}\label{density-decomposition}
	\rho_\gamma(x)
	=
	\int_0^1\rho_{\gamma_t}(x)\,dt\ \ \ \text{a.e. in} \,\  \R^d.
\end{align}
\end{lem}

\noindent \textbf{Proof.} For any $\gamma\in\mathcal{K}_2$, it follows from \eqref{3.16} that there exist an orthonormal system  $\{u_i\}_{i\geq1}$  in $ L^2(\R^d, \C)$ and a sequence $\{q_i\}_{i\geq1}\subset\R$ such that
\begin{equation*}
\gamma=2\sum_{i\geq1}
q_i|u_i\rangle\langle u_i|, \ \   \text{where}\ \,  	0<q_i\leq\frac12\,\ \text{and}\,\  \sum_{i\geq1}q_i=1.
\end{equation*}
We identify $[0,1)$ with the one-dimensional torus $\mathbb{T}:=\R/\mathbb{Z}$, and choose mutually disjoint half-open	 intervals $\{I_i\}_{i\geq1}$ such that
\[
\mathbb{T}
=
\bigcup_{i\geq1}I_i,
\ \
|I_i|=q_i.
\]
Define
\begin{align}\label{3.19}
\tau(t):=t+\frac12 \pmod 1.
\end{align}
Since $|I_i|\leq1/2$, we have
\begin{align}\label{3.20a}
I_i\cap\tau(I_i)=\varnothing\ \ \text{up to a set of measure zero},\ \   i=1, 2, \cdots.
\end{align}
For almost every $t\in[0,1)$, let $i(t)\in \mathbb{N}^+$ be the unique index such that
$t\in I_{i(t)}$, and define
\begin{align}\label{3.20}
\gamma_t:=|u_{i(t)}\rangle\langle u_{i(t)}|+|u_{i(\tau(t))}\rangle\langle u_{i(\tau(t))}|.
\end{align}
By the definition of $i(t)$, we deduce from \eqref{3.20a} that $i(t)\neq i(\tau(t))$  holds for almost every $t\in[0,1)$. This then yields that  the operator $\gamma_t$ is a rank-two orthogonal projection for almost every $t\in[0,1)$, and hence $\gamma_t\in\mathcal{P}_2$ holds for almost every $t\in[0,1)$.

Let $\gamma_t\in\mathcal{P}_2$ be given by \eqref{3.20}. It then follows from \eqref{1.4b} that
\[
\Tr(-\Delta \gamma_t)=\big\langle u_{i(t)}, -\Delta u_{i(t)}\big\rangle+\big\langle u_{i(\tau(t))}, -\Delta u_{i(\tau(t))}\big\rangle,
\]
which further implies that
\begin{align}\label{3.23}
\int_0^1\Tr(-\Delta \gamma_t)\,dt
=&\int_0^1\big\langle u_{i(t)}, -\Delta u_{i(t)}\big\rangle dt
+ \int_0^1\big\langle u_{i(\tau(t))}, -\Delta u_{i(\tau(t))}\big\rangle dt.
\end{align}
Since $i(t)=i$ holds for almost every  $t\in I_i\cap[0, 1)$, we have
\begin{align}\label{3.22}
\int_0^1\big\langle u_{i(t)}, -\Delta u_{i(t)}\big\rangle dt
=&\sum_{i\geq1}\int_{I_i\cap[0, 1)} \big\langle u_{i(t)}, -\Delta u_{i(t)}\big\rangle dt\\
=&\sum_{i\geq1}\int_{I_i\cap[0, 1)} \big\langle u_{i}, -\Delta u_{i}\big\rangle dt
=\sum_{i\geq1}|I_i|\big\langle u_{i}, -\Delta u_{i}\big\rangle.\nonumber
\end{align}
We also deduce from \eqref{3.19} that
\begin{align}\label{3.24}
\int_0^1\big\langle u_{i(\tau(t))}, -\Delta u_{i(\tau(t))}\big\rangle dt
=&\int_0^{1/2}\big\langle u_{i(\tau(t))}, -\Delta u_{i(\tau(t))}\big\rangle dt+\int_{1/2}^1\big\langle u_{i(\tau(t))}, -\Delta u_{i(\tau(t))}\big\rangle dt\nonumber\\
=&\int_{1/2}^1\big\langle u_{i(s)}, -\Delta u_{i(s)}\big\rangle ds+\int_0^{1/2}\big\langle u_{i(s)}, -\Delta u_{i(s)}\big\rangle ds\\
=&\int_0^1\big\langle u_{i(s)}, -\Delta u_{i(s)}\big\rangle ds
=\sum_{i\geq1}|I_i|\big\langle u_{i}, -\Delta u_{i}\big\rangle.\nonumber
\end{align}
We then deduce from \eqref{3.23}--\eqref{3.24} that
\begin{align}\label{3.26}
\int_0^1\Tr(-\Delta \gamma_t)\,dt
=\sum_{i\geq1}2|I_i|\big\langle u_{i}, -\Delta u_{i}\big\rangle
=\sum_{i\geq1}n_i\big\langle u_{i}, -\Delta u_{i}\big\rangle
=\Tr(-\Delta\gamma),
\end{align}
which thus proves \eqref{kinetic-decomposition}. Similar to \eqref{3.26}, one can verify that  for almost every $x\in\R^d$,
\begin{align*}
\int_0^1\rho_{\gamma_t}(x)\,dt
=&\int_0^1|u_{i(t)}(x)|^2dt
+\int_0^1|u_{i(\tau(t))}(x)|^2dt
\\
=&\sum_{i\geq1}2|I_i|\, |u_i(x)|^2
=\sum_{i\geq1}n_i|u_i(x)|^2
=\rho_\gamma(x),
\end{align*}
which proves \eqref{density-decomposition}. This therefore completes the proof of  Lemma \ref{lem-projector-decomposition}.\qed

\vspace{0.1 cm}
Recall that the minimization problem $e_2^*=\inf\limits_{\gamma\in \mathcal{P}_2}\mathcal{E}(\gamma)$  is defined by \eqref{e}.   
The following lemma shows that $e_2^*$ keeps the same as in \eqref{e} if  the constraint   $\mathcal{P}_2$ of \eqref{e}  is replaced by $\mathcal{K}_2$.

\begin{lem}\label{lemrelax}
Let $e_2^*>0$ and  $\mathcal{K}_2$ be defined by \eqref{e} and \eqref{k2}, respectively. Then we have
\begin{align*}
	0<e_2^*=\inf_{\gamma\in\mathcal{K}_2}\mathcal{E}(\gamma).
\end{align*}
\end{lem}

\noindent \textbf{Proof.}
The same argument of \eqref{3.83} gives  that
\begin{align}\label{3.29a}
\inte \big(\rho_\gamma\text{log}\rho_\gamma\big)_+ dx
<\infty,\ \    \forall\  \gamma\in\mathcal{K}_2.
\end{align}
This implies that $\mathcal{E}: \mathcal{K}_2\to\R\cup \{+\infty\}$ is well-defined. 
Since $\mathcal{P}_2\subset\mathcal{K}_2$,  we have
\begin{align}\label{easy-relaxed-bound}
\inf_{\gamma\in\mathcal{K}_2}\mathcal{E}(\gamma)
\leq
\inf_{\gamma\in\mathcal{P}_2}\mathcal{E}(\gamma)
=e_2^*.
\end{align}

It remains to prove the reverse inequality of (\ref{easy-relaxed-bound}).
Let $\gamma\in\mathcal{K}_2$.  If $\inte \rho_\gamma\text{log}\rho_\gamma dx=-\infty$, then $\mathcal{E}(\gamma)=+\infty$, and we are done. Thus, without loss of generality, we may assume from \eqref{3.29a} that  $-\infty<\inte \rho_\gamma\text{log}\rho_\gamma dx<\infty$. 
By Lemma \ref{lem-projector-decomposition}, then there exists a measurable
family $\{\gamma_t\}_{t\in[0,1)}\subset\mathcal{P}_2$ such that
\begin{align}\label{3.29}
\Tr(-\Delta\gamma)=\int_0^1\Tr(-\Delta \gamma_t)\,dt,
\end{align}
and
\begin{align}\label{3.28}
\rho_\gamma(x)=\int_0^1\rho_{\gamma_t}(x)dt\ \ \ \text{for almost every}\ \, x\in\R^d.
\end{align}
The same argument of  \eqref{3.29a} yields from \eqref{3.29} that
\begin{equation*}
\begin{split}\int_0^1\int_{\R^d}
	\big(\rho_{\gamma_t}\log\rho_{\gamma_t}\big)_+\,dx\,dt
	\leq&
	C_d c^{-1}_{LT}(d)
	\int_0^1\Tr(-\Delta \gamma_t)\,dt\\
	=&C_dc^{-1}_{LT}(d)\Tr(-\Delta\gamma)<\infty,
\end{split}\end{equation*}
which thus gives that $\int_0^1\int_{\R^d}\rho_{\gamma_t}\log\rho_{\gamma_t}\,dxdt$ is well-defined.

We now claim that
\begin{align}\label{entropy-jensen}
-\infty<\int_{\R^d}\rho_\gamma\log\rho_\gamma\,dx
\leq\int_0^1\int_{\R^d}\rho_{\gamma_t}\log\rho_{\gamma_t}\,dxdt<\infty.
\end{align}
Indeed, define for $\varepsilon\in(0,1]$,
\[
h_\varepsilon(s):=s\log(s+\varepsilon), \ \    s\geq0.
\]
Since
\[
h_\varepsilon''(s)
=
\frac{s+2\varepsilon}{(s+\varepsilon)^2}>0, \ \    s\geq0,
\]
we obtain that  $h_\varepsilon (s)$ is convex in $s\geq0$. By Jensen's inequality, this then implies from \eqref{3.28}
that
\begin{align}\label{3.35c}
h_\varepsilon(\rho_\gamma(x))
=
h_\varepsilon\left(\int_0^1\rho_{\gamma_t}(x)\,dt\right)
\leq
\int_0^1h_\varepsilon\big(\rho_{\gamma_t}(x)\big)\,dt\ \,\ \mbox{a.e. in}\ \, \R^d.
\end{align}
Note that for any $s\geq0$ and $\varepsilon\in(0,1]$,
\[h_\varepsilon(s)\geq s\log\varepsilon,\]and
\[
\big(h_\varepsilon(s)\big)_+
\leq s\log(1+s)
\leq \tilde C_d\big(s+s^{1+\frac2d}\big),
\]
where $\tilde C_d>0$ depends  only on $d\in\mathbb{N}^+$. Following \eqref{lt}, these further yield from   \eqref{3.29} and \eqref{3.28} that for any  $\varepsilon\in(0,1]$,
\begin{align}\label{3.36b}
0&\leq\int_{\R^d}\big(h_\varepsilon(\rho_\gamma)\big)_-dx
\leq \big(-\log\varepsilon\big)\int_{\R^d} \rho_\gamma dx<\infty,\\[2mm]
0&\leq\int_{\R^d}\big(h_\varepsilon(\rho_\gamma)\big)_+dx
\leq \tilde C_d\int_{\R^d} \rho_\gamma dx+\tilde C_d\int_{\R^d} \rho_\gamma^{1+\frac2d} dx<\infty,\\[2mm]\label{3.36d}
0&\leq\int_{\R^d}\int_0^1\big(h_\varepsilon(\rho_{\gamma_t})\big)_-dtdx
\leq  \big(-\log\varepsilon\big)\int_{\R^d}\int_0^1 \rho_{\gamma_t}dtdx\nonumber\\
& \ \ \ \ \ \ \ \ \ \  \  \ \ \  \ \ \ \ \ \ \ \ \ \ \ \ \ \ \ \ \ \ \ \  =\big(-\log\varepsilon\big)\int_{\R^d}\rho_{\gamma}dx<\infty,
\end{align}	
and
\begin{align}\label{3.37}
0\leq \int_{\R^d}\int_0^1\big(h_\varepsilon(\rho_{\gamma_t})\big)_+dtdx 
\leq &\tilde C_d\int_{\R^d} \int_0^1\rho_{\gamma_t} dtdx
+\tilde C_d\int_{\R^d}\int_0^1 \rho_{\gamma_t}^{1+\frac2d} dtdx\nonumber\\
=&\tilde C_d\inte\rho_\gamma dx+\tilde C_d\int_0^1\int_{\R^d} \rho_{\gamma_t}^{1+\frac2d} dxdt\\
\leq& \tilde C_d\inte\rho_\gamma dx+\tilde C_dc_{LT}^{-1}(d)\int_0^1\text{Tr}(-\Delta \gamma_t)dt\nonumber\\
=&\tilde C_d\inte\rho_\gamma dx+\tilde C_dc_{LT}^{-1}(d)\text{Tr}(-\Delta \gamma)	<\infty.\nonumber
\end{align}

By Fubini's theorem, we  conclude from \eqref{3.35c}--\eqref{3.37} that for any  $\varepsilon\in(0,1]$,
\begin{align}\label{regularized-jensen}
-\infty<\int_{\R^d}h_\varepsilon(\rho_\gamma)\,dx
\leq\int_{\R^d}\int_0^1h_\varepsilon(\rho_{\gamma_t})\,dt\,dx
=\int_0^1\int_{\R^d}h_\varepsilon(\rho_{\gamma_t})\,dx\,dt<\infty.
\end{align}
Since
\[
h_1(s)-h_\varepsilon(s)\geq0,\ \   \forall\ s\geq0,
\]
and $h_1(s)-h_\varepsilon(s)$ increases pointwise in $s\geq0$ as $\varepsilon\searrow0$, the monotone convergence theorem gives that
\[
\lim_{\varepsilon\searrow0}
\int_{\R^d}h_\varepsilon(\rho_\gamma)\,dx
=\int_{\R^d}\rho_\gamma\log\rho_\gamma\,dx.
\]
Applying the same argument  as above on $[0,1)\times\R^d$, one gets that
\[
\lim_{\varepsilon\searrow0}
\int_0^1\int_{\R^d}h_\varepsilon(\rho_{\gamma_t})\,dx\,dt
=
\int_0^1\int_{\R^d}\rho_{\gamma_t}\log\rho_{\gamma_t}\,dx\,dt.
\]
Setting $\varepsilon\searrow0$ for \eqref{regularized-jensen},  we therefore obtain that the claim \eqref{entropy-jensen} holds true.



As a consequence of \eqref{3.29} and \eqref{entropy-jensen}, we have
\begin{align*}
\mathcal{E}(\gamma)
&=\Tr(-\Delta\gamma)-\int_{\R^d}\rho_\gamma\log\rho_\gamma\,dx\\
&\geq
\int_0^1
\left[
\Tr(-\Delta \gamma_t)
-
\int_{\R^d}\rho_{\gamma_t}\log\rho_{\gamma_t}\,dx
\right]dt
=\int_0^1\mathcal{E}(\gamma_t)\,dt.
\end{align*}
Since $\gamma_t\in\mathcal{P}_2$ holds for almost every $t\in[0,1)$,    the definition of $e_2^*$ gives that  for almost every $t\in[0,1)$,
\[
\mathcal{E}(\gamma_t)\geq e_2^*,
\ \
\mathcal{E}(\gamma)
\geq
\int_0^1e_2^*\,dt
=e_2^*.
\]
Taking the infimum over $\gamma\in\mathcal{K}_2$, it then yields from above that  $\inf_{\gamma\in\mathcal{K}_2}\mathcal{E}(\gamma) \geq e_2^*.$
Together with \eqref{easy-relaxed-bound}, this therefore completes the proof of Lemma \ref{lemrelax}. \qed


\vspace{.15cm}

We are next ready to prove Theorem \ref{thmlog}, which is concerned with the sharp logarithmic Sobolev inequality \eqref{logsobo} for orthonormal functions. 

\vspace{.15cm}

\noindent\textbf{Proof of Theorem \ref{thmlog}.} Applying Theorem \ref{thm1}, we first derive from  Lemma \ref{lemrelax} that
\begin{align}\label{4.65m}
\Tr(-\Delta\gamma)-\int_{\R^d}\rho_\gamma\log\rho_\gamma\,dx\geq e_2^*>0, \ \   \forall\ \gamma\in\mathcal{K}_2,
\end{align}
and the optimal constant $e_2^*$ can be attained for some $\gamma^*\in\mathcal{P}_2\subset \mathcal{K}_2$.
We first claim that $\Tr(-\Delta\gamma)>0$ holds for   any $\gamma\in \mathcal{K}_2$. Indeed, if $\Tr(-\Delta\gamma)=0$ holds for some $\gamma\in \mathcal{K}_2$, then each eigenfunction of $\gamma$ corresponding to a positive eigenvalue
would have zero gradient and hence it would be a  constant. This is however
impossible, and the above claim is hence proved.

Define
$$
\gamma_\lambda(x, y):=\lambda^d\gamma(\lambda x, \, \lambda y),\ \   \lambda>0,
$$
where  $\gamma_\lambda(x, y)$ and $\gamma(x, y)$ denote the integral kernels of  the operators $\gamma_\lambda$ and $\gamma$, respectively. We then have $\gamma_\lambda\in\mathcal{K}_2$ and
\begin{align}\label{M:3.42}
e_2^*\leq\mathcal{E}(\gamma_\lambda)=\lambda^2\Tr(-\Delta\gamma)
-\int_{\R^d}\rho_\gamma\log\rho_\gamma\,dx
-2d\log\lambda,\ \   \forall\ \lambda>0,
\end{align}
where \eqref{4.65m} is used. Note that the infimum of the right-hand side for (\ref{M:3.42}) is attained at $\lambda^2=\frac{d}{\Tr(-\Delta\gamma)}$. This thus yields that 
\begin{align}\label{3.42}
\int_{\R^d}\rho_\gamma\log\rho_\gamma\,dx
\leq
d\log\left(\frac{\Tr(-\Delta\gamma)}{d}\right)
+d-e_2^*,\ \   \forall\ \gamma\in \mathcal{K}_2,
\end{align}
which proves \eqref{logsobo}.

Let  $\gamma^*\in \mathcal{K}_2$ be a minimizer of $e_2^*$. Then the function
\[
\lambda\longmapsto\mathcal{E}(\gamma^*_\lambda)=\lambda^2\Tr(-\Delta\gamma^*)
-\int_{\R^d}\rho_{\gamma^*}\log\rho_{\gamma^*}dx
-2d\log\lambda
\]
attains its minimum at $\lambda=1$. This further yields that
\[
0=\left.\frac{d}{d\lambda}\mathcal{E}(\gamma^*_\lambda)\right|_{\lambda=1}
=2\Tr(-\Delta\gamma^*)-2d,
\]
and hence
\[
\Tr(-\Delta\gamma^*)=d,\ \   \int_{\R^d}\rho_{\gamma^*}\log\rho_{\gamma^*}dx
=d-e_2^*.
\]
Therefore, the identity of \eqref{3.42} holds for any minimizer of $e_2^*$, and the proof of Theorem \ref{thmlog} is therefore complete. \qed

Let $e_2^*>0$ and $\mathcal{K}_2$ be given by \eqref{e} and \eqref{k2}, respectively. We finally remark the following equivalence.

\begin{rem}\label{rem}
The proof of Theorem \ref{thmlog} implies that any minimizer of $e_2^*$ is essentially an optimizer of \eqref{logsobo}. Conversely, one can also conclude that up to an appropriate normalization,
any optimizer  of \eqref{logsobo} is actually a minimizer of $e_2^*=\inf_{\gamma\in\mathcal{K}_2}\mathcal{E}(\gamma)$.
To address the latter one, let $\gamma\in\mathcal{K}_2$ be an optimizer of \eqref{logsobo}. Since the identity of \eqref{3.42} holds essentially for $\gamma$, we have
\begin{align}\label{r}
\int_{\R^d}\rho_\gamma\log\rho_\gamma\,dx
		= d\log\left(\frac{\Tr(-\Delta\gamma)}{d}\right)+d-e_2^*.
\end{align}
The proof of Theorem \ref{thmlog} also gives that $\text{Tr}(-\Delta\gamma)>0$. Denote $\gamma_\lambda(x, y):=\lambda^d\gamma(\lambda x, \lambda y)$, where $\lambda=\big(d/\text{Tr}(-\Delta\gamma)\big)^{1/2}$. We then derive from \eqref{r} that
	\begin{align*}
		\mathcal{E}(\gamma_\lambda)
		=&\lambda^2\Tr(-\Delta\gamma)
		-\int_{\R^d}\rho_\gamma\log\rho_\gamma\,dx-2d\log\lambda\\
		=&d-\int_{\R^d}\rho_\gamma\log\rho_\gamma\,dx+d\, \text{log}\big(\frac{\text{Tr}(-\Delta\gamma}{d}\big)
		=e^*_2.
	\end{align*}
Together with Lemma \ref{lemrelax}, this further implies that $\gamma_\lambda$ is actually a minimizer of $e_2^*=\inf_{\gamma\in\mathcal{K}_2}\mathcal{E}(\gamma)$, and we are thus done.
\end{rem}

\noindent{\bf Acknowledgements:} The authors thank Prof. Phan Th\`anh Nam very much for his fruitful discussions on the present paper.




\begin{thebibliography}{99}
	
	
\bibitem{zfl} X. M. An, S. J.  Peng, X. Yang and F. L. Zhong, {\em Qualitative analysis for ground state solutions of logarithmic Schr\"odinger equations under a small constant magnetic field in $\R^N$}, J. Funct. Anal. {\bf 290} (2026), no. 7, 111333.

\bibitem{regu1}J. Angulo Pava and A. H. Ardila, {\em  Stability of standing waves for the logarithmic Schr\"{o}dinger equation with attractive delta potential},  Indiana Univ. Math. J. {\bf 67} (2018), no. 2, 471--494.





\bibitem{begain} M. Badiale and E. Serra, Semilinear Elliptic Equations for Beginners, Springer, London, 2011.

\bibitem{lognls} I. Bialynicki-Birula and  J. Mycielski, {\em Nonlinear wave mechanics}, Ann. Phys. {\bf100} (1976), 62--93.

	
\bibitem{CarlenLieb}
E. A. Carlen and E. H. Lieb, \emph{Optimal hypercontractivity for Fermi fields and related non-commutative integration inequalities}, Comm. Math. Phys. \textbf{155} (1993), no. 1, 27--46.	
	
	
\bibitem{1983NA} T. Cazenave, {\em Stable solutions of the logarithmic Schr\"odinger equation,} Nonlinear Anal. {\bf 7} (1983), no. 10, 1127--1140.
	
	
\bibitem{Cazenave} T. Cazenave,  Semilinear Schr\"{o}dinger Equations, Courant Lecture Notes in Mathematics Vol. {\bf10}, Courant Institute of Mathematical Science, New York, 2003.
	
\bibitem{cg1} B. Chen and Y. J. Guo, {\em Ground states of fermionic nonlinear Schr\"odinger systems with Coulomb potential I: the $L^2$-subcritical case},  Lett. Math. Phys. {\bf 114} (2024), no. 6, Paper No. 132.

\bibitem{ims} H. L. Cycon, R. G. Froese, W. Kirsch and B. Simon,  Schr\"{o}dinger Operators with Application to Quantum Mechanics and Global Geometry, Texts and Monographs in Physics, Springer Study Edition, Springer-Verlag, Berlin, 1987.

\bibitem{uniq2014} P. D'Avenia, E. Montefusco  and M. Squassina, {\em On the logarithmic Schr\"odinger equation}, Commun. Contemp. Math. {\bf 16} (2014), 1350032.
	
\bibitem{depino} M. Del Pino and  J. Dolbeault, {\em The optimal Euclidean  $L^p$-Sobolev logarithmic inequality}, J. Funct. Anal. {\bf 197} (2003), no. 1, 151--161.
	
	
\bibitem{DFLP} J. Dolbeault, P. Felmer, M.  Loss and  E.  Paturel,  {\em Lieb-Thirring type inequalities and Gagliardo-Nirenberg inequalities for systems}, J. Funct. Anal. {\bf 238} (2006), no. 1, 193--220.
	
\bibitem{clt} J. Dolbeault, A. Laptev and M. Loss, {\em Lieb-Thirring inequalities with improved constants}, J. Eur. Math. Soc. {\bf 10} (2008), 1121--1126.

\bibitem{open} R. L. Frank,  {\em The Lieb-Thirring inequalities: recent results and open problems}, Proc. Sympos. Pure Math. {\bf 104} (2021), 45--86.
	
\bibitem{ii} R. L. Frank, D. Gontier and M. Lewin, {\em The nonlinear Schr\"{o}dinger equation for orthonormal functions II: Application to Lieb-Thirring inequalities}, Comm. Math. Phys. {\bf 384} (2021), no. 3, 1783--1828.

\bibitem{elli} D. Gilbarg and N. S. Trudinger, Elliptic Partial Differential Equations of Second Order, Classics in Math. Springer-Verlag, Berlin, 2001.


\bibitem{i} D. Gontier, M. Lewin and F. Q. Nazar, \emph{The nonlinear Schr\"odinger equation for orthonormal functions: existence of ground states}, Arch. Ration. Mech. Anal. {\bf240} (2021), 1203--1254.

\bibitem{Gross}
L. Gross, \emph{Hypercontractivity and logarithmic Sobolev inequalities for the Clifford--Dirichlet form}, Duke Math. J. \textbf{42} (1975), no. 3, 383--396.


\bibitem{hq} Q. Han and F. H. Lin, Elliptic Partial Differential Equations, 2nd ed., Courant Lecture Notes in Math. Vol. {\bf 1}, Courant Institute of Math. Science/AMS, New York, 2011.


\bibitem{ho} M. Hoffmann-Ostenhof and T. Hoffmann-Ostenhof,  \emph{Schr\"{o}dinger inequalities and asymptotic behavior of the electron density of atoms and molecules}, Phys. Rev. A  {\bf 16} (1977), 1782--1785.





\bibitem{ge} M. Lewin, {\em Geometric methods for nonlinear many-body quantum systems}, J. Funct. Anal. {\bf260} (2011), no. 12, 3535--3595.


\bibitem{Lieb1981} E. H. Lieb, \emph{Variational principle for many-fermion systems}, Phys. Rev. Lett. \textbf{46} (1981), no. 7, 457--459.


\bibitem{analysis} E. H. Lieb and M. Loss,  Analysis, Graduate Studies in Mathematics Vol. 14, 2nd ed., American Mathematical Society, Providence, RI, 2001.









\bibitem{Rabinowitz} P. H. Rabinowitz, \emph{On a class of nonlinear Schr\"{o}dinger equations}, Z. Angew. Math. Phys. {\bf 43} (1992), 270--291.

\bibitem{kato} M. Reed and B. Simon,  Methods of Modern Mathematical Physics II: Fourier analysis, self-adjointness, Academic Press, New York-London, 1975.

\bibitem{modern4} M. Reed and B. Simon, Methods of Modern Mathematical Physics, IV: Analysis of Operators, Academic Press, New York-London, 1978, xv+396 pp.


\bibitem{shuai} W. Shuai, {\em Multiple solutions for logarithmic Schr\"odinger equations}, Nonlinearity {\bf 32} (2019), no. 6, 2201--2225.






\bibitem{wzqref26} W. C. Troy, {\em Uniqueness of positive ground state solutions of the logarithmic Schr\"odinger equation}, Arch. Ration. Mech. Anal. {\bf 222} (2016), 1581--1600.



\bibitem{wzq} Z. Q. Wang and C. X. Zhang, {\em  Convergence from power-law to logarithm-law in nonlinear scalar field equations},  Arch. Ration. Mech. Anal. {\bf 231} (2019), no. 1, 45--61.

\bibitem{1987log} F. B. Weissler, {\em Logarithmic Sobolev inequalities for the heat-diffusion semigroup}, Trans. Amer. Math. Soc. {\bf237} (1978), 255--269.

\bibitem{minimax} M. Willem,  Minimax Theorems, Progress in Nonlinear Differential Equations and Their Applications Vol. {\bf 24}, Birkh\"{a}user Boston, Inc., Boston, 1996.



\bibitem{normal} L. Y. Zhang and C. X. Zhang, {\em The asymptotic behaviors of normalized ground states for nonlinear Schr\"odinger equations},
NoDEA Nonlinear Differential Equations Appl. {\bf30} (2023), no. 3, Paper No. 44, 12 pp.



\end{thebibliography}
\end{document}